\documentclass[numbook,envcountsame,smallextended]{svjourN}

\usepackage{pb-diagram,amsfonts,amsmath}

\newcommand\D{\mathbb{D}}
\newcommand\N{\mathbb{N}}
\newcommand\R{\mathbb{R}}
\newcommand\C{\mathbb{C}}

\renewcommand\H{\mathcal{H}}
\newcommand\T{\mathcal{T}}
\newcommand\M{\mathcal{M}}
\renewcommand\S{\mathcal{S}}

\newcommand\AT{\overline{\mathcal{T}}}
\newcommand\AM{\overline{\mathcal{M}}}
\renewcommand\P{\mathbb{P}}
\newcommand\Sphere{{\mathbb{S}^2}}

\newcommand{\id}{\mbox{\rm id}}
\renewcommand{\mod}{\mbox{\rm mod }}

\newcommand\mh{\hat{m}}
\renewcommand\th{\hat{\tau}}

\renewcommand\tt{\tilde{\tau}}
\newcommand{\hh}{\tilde{h}}
\newcommand{\sigmat}{\tilde{\sigma}}

\newcommand{\sm}{\setminus}
\newcommand{\eps}{\varepsilon}

\spnewtheorem*{proofof}{Proof}{\bf}{\rm}

\def\Teich{Teich\-m\"uller }

\begin{document}

\title{Thurston's pullback map on the augmented Teichm\"uller space and applications}
\author{Nikita Selinger}
\institute{Jacobs University Bremen, Campus Ring 1, 28759 Bremen, Germany\\\email{n.selinger@jacobs-university.de}}

\date{September 13, 2011}
\maketitle

\begin{abstract}

Let $f$ be a postcritically finite branched self-cover of a 
2-dimensional topological sphere. Such a map induces an analytic self-map 
$\sigma_f$ of a finite-dimensional Teichm\"uller space. We 
prove that this map extends continuously to the  
augmented Teichm\"uller space and give an explicit construction for 
this extension. This allows us to characterize the dynamics of 
Thurston's pullback map near invariant strata of the boundary of the 
augmented Teichm\"uller space. The resulting classification of 
invariant boundary strata is used to prove a conjecture by Pilgrim 
and to infer further properties of Thurston's pullback map. Our 
approach also yields new proofs of Thurston's theorem and Pilgrim's 
Canonical Obstruction theorem.
\end{abstract}

\tableofcontents

\section{Introduction}

     In the early 1980's Thurston proved one of the most
important theorems in the field of Complex Dynamics. His
characterization theorem provides a topological criterion of whether a given
combinatorics can be realized by a rational map, and also provides
a rigidity statement: two rational maps are equivalent if and only
if they are conjugate by a Moebius transformation.

     The original proof from \cite{DH} of Thurston's characterization theorem (Theorem~\ref{thm:Thurston}) 
     relates the original question to whether or not Thurston's pullback map on a \Teich space has a fixed point.
In \cite{DH} the authors study the behavior of Thurston's pullback map near infinity without specifying any structure at the boundary of the considered \Teich space. The question of finding the notion of the boundary for the \Teich space that would be appropriate for the problem is very natural, and was in the air since the first proof of the theorem came out.
     
     One obvious candidate to consider is the Thurston boundary which has been successfully used by Thurston to give a similar but considerably simpler proof of the characterization theorem for surface diffeomorphisms. The analytic self-map of the \Teich space investigated in this case extends continuously to the Thurston boundary of the \Teich space. This yields a continuous self-map of a topological ball. The proof then uses Brouwer's fixed point theorem as an essential ingredient.

     In the case when the \Teich space is one dimensional, Thurston's pullback map is a self-map of the unit disk $\D$. One immediately notices that for the map $f(z)=3z^2/(2z^3+1)$ considered in \cite{BEKP}, Thurston's pullback map can not be continued to a self-map of $\overline{\D}$ which is naturally homeomorphic to the Thurston compactification of the \Teich space. Note that in some cases, such an extension is possible. For instance, in the same article \cite{BEKP} the authors present an example by McMullen of branched covers with constant Thurston's pullback map; other examples are Latt\`es maps for which the pullback maps are automorphisms of $\D$. In Section~\ref{sec:ThurstonBoundary} we give a more conceptual construction of postcritically finite branched covers $f$ such that Thurston's pullback map does not extend to the Thurston boundary of the \Teich space (see Theorem~\ref{thm:noextension}).
     
    From our point of view, Theorem~\ref{thm:noextension} just says that the Thurston boundary is not the right boundary notion for the task. The next theorem has significantly more consequences and is in the heart of the whole article.
  
\begin{theorem} 
\label{thm:extensionA}
Thurston's pullback map extends continuously to a self-map of the augmented \Teich space.
\end{theorem}

     The topology of the augmented \Teich space is by far more complicated than the topology of the  compactification of the \Teich space with the Thurston boundary. It is not compact or even locally compact, so that we can not apply tools like Brouwer's fixed point theorem. It is not true that Thurston's pullback map must always have a fixed point in the augmented \Teich space (the simplest counterexample is a Latt\`es map corresponding to a matrix with distinct real eigenvalues, see \cite{DH} for a definition), however, this is true in many cases (see Theorem~\ref{thm:fixedpts}).
     
     In Section~\ref{sec:extension} we define the extension to the boundary of the augmented \Teich space explicitly in a way that is similar to the definition of the action of Thurston's pullback map on the \Teich space. This brings, in our opinion, new insights in understanding the behavior of Thurston's pullback map. In Section~\ref{sec:strata} we characterize the dynamics of Thurston's pullback map near invariant strata on the boundary of the augmented \Teich space. In Section~\ref{sec:proofs} we use the obtained classification to simplify the proofs of Thurston's theorem and Canonical Obstruction theorem due to Pilgrim (see Theorem~\ref{thm:CanonicalObstruction}).
     
     In Section~\ref{sec:Pilgrim}  
     an application of our approach is given: we prove a conjecture by Kevin Pilgrim \cite{P1} (see Theorem~\ref{thm:KevinConjecture}).

\section{Basic definitions}

The main setup is the same as in \cite{DH}.

Consider  an orientation-preserving branched self-cover $f$ of degree $d_f \ge2$ of the 2-dimensional topological sphere $\Sphere$. The \emph{critial set} $\Omega_f$ is the set of all points $z$ in $\Sphere$ where the local degree of $f$ is greater than 1. The \emph{postcritial set} $P_f$ is the union of all forward orbits of $\Omega_f$ : $P_f= \cup_{i\ge1} f^i(\Omega_f).$  A branched cover $f$ is called \emph{postcritially finite} or a \emph{Thurston map} if $P_f$ is finite. Denote $p_f= \#P_f$.

\begin{remark} 
\label{rk:GenDef} We can generalize the definition of Thurston maps as follows. A pair $(f,P)$ of a branched cover $f \colon \Sphere \to \Sphere$ and a finite set $P \subset \Sphere$ is called a Thurston map if $P$ is forward invariant and contains all critical values of $f$ (and, hence, contains $P_f$). All results discussed in the present article hold for this general definition of Thurston maps.
\end{remark}

Two Thurston maps $f$ and $g$ are \emph{Thurston equivalent} if and only if there exist two homeomorphisms $h_1,h_2\colon \Sphere \to \Sphere$ such that the diagram
\[
\begin{diagram}
\node{(\Sphere,P_f)} \arrow{e,t}{h_1} \arrow{s,l}{f} \node{(\Sphere,P_g)}
\arrow{s,l}{g}
\\
\node{(\Sphere,P_f)} \arrow{e,t}{h_2} \node{(\Sphere,P_g)}
\end{diagram}
\]
commutes, $h_1|_{P_f}=h_2|_{P_f}$, and
$h_1$ and $h_2$ are homotopic relative to $P_f$.

A simple closed curve $\gamma$ is called \emph{essential} if every component of $\Sphere \sm \gamma$ contains at least two points of $P_f$. We consider essential simple closed curves up to free homotopy in $\Sphere\sm P_f$. A \emph{multicurve} is a finite set of pairwise disjoint and non-homotopic essential simple closed curves. Denote by $f^{-1}(\Gamma)$ the multicurve consisting of all essential preimages of curves in $\Gamma$. A multicurve $\Gamma=(\gamma_1,\ldots,\gamma_n)$ is called \emph{invariant} if each
component of $f^{-1}(\gamma_i)$ is either non-essential, or it is
homotopic (in $\Sphere\sm P_f$) to a curve in $\Gamma$ (i.e. $f^{-1}(\Gamma) \subseteq \Gamma$).
We say that $\Gamma$ is \emph{completely invariant} if $f^{-1}(\Gamma) = \Gamma$. 

Every
invariant multicurve $\Gamma$ has its associated \emph{Thurston
matrix} $M_\Gamma=(m_{i,j})$ with
\[
m_{i,j}=\sum_{\gamma_{i,j,k}} (\deg f|_{\gamma_{i,j,k}}\colon
\gamma_{i,j,k}\to \gamma_j)^{-1}
\]
where $\gamma_{i,j,k}$ ranges  through all preimages of $\gamma_j$
that are homotopic to $\gamma_i$. Since all entries of $M_\Gamma$
are non-negative real, the leading eigenvalue $\lambda_\Gamma$ of
$M_\Gamma$ is real and non-negative. The multicurve $\Gamma$ is a
\emph{Thurston obstruction} if $\lambda_\Gamma\ge 1$.

We call $\Gamma$ a \emph{simple} obstruction (compare \cite{P}) if no permutation of the curves in $\Gamma$ puts $M_\Gamma$ in the block form
$$ M_\Gamma = \left( 
\begin{array}{cc}
	M_{11}  & 0 \\
	M_{21}  & M_{22}
\end{array}
 \right),
$$ 
where  the leading eigenvalue of $M_{11}$ is less than $1$. If such a permutation exists, it follows that $M_{22}$ is a Thurston matrix of an invariant multicurve with the same leading eigenvalue as $M_\Gamma$. It is thus evident that every obstruction contains a simple one. Every simple obstruction is automatically completely invariant. For two real vectors $a=(a_1,\ldots,a_n)^T$ and $b=(b_1,\ldots,b_n)^T$, we write $a \ge b$ (or $a > b$) if  $a_i \ge b_i$ (respectively, $a_i > b_i$) for all $i=\overline{1,n}$. The following is an exercise in linear algebra.

\begin{proposition}
\label{prop:positive}
  An invariant multicurve $\Gamma$ is a simple obstruction if and only if there exists a  vector $v>0$ such that $M_\Gamma v \ge v$.
\end{proposition}

Thurston's original characterization theorem is formulated as follows:

\begin{theorem}[Thurston's Theorem \cite{DH}]
\label{thm:Thurston}  A postcritically finite  bran\-ched cover
$f\colon\Sphere\to\Sphere$ with hyperbolic orbifold is either
Thurston-equivalent to a rational map $g$ (which is then
necessarily unique up to conjugation by a Moebius transformation), or
$f$ has a Thurston obstruction.
\end{theorem}
General rigorous definition of orbifolds and their Euler
characteristic can be found in \cite{M}. In our case, there is a unique and straightforward way to construct the minimal
function $v_f$ of all functions  $v:\Sphere \to \N \cup \{\infty\}$
satisfying the following two conditions:
\begin{enumerate}
 \item[(i)] $v(x)=1$ when $x \notin P_f$;
 \item[(ii)] $v(x)$ is divisible by $v(y) \deg_y f$ for all  $y \in f^{-1}(x)$.
\end{enumerate}
We say that $f$ has a hyperbolic orbifold $O_f=(\Sphere,v_f)$ if the
Euler characteristic of $O_f$
\[\chi(O_f)=2- \sum_{x \in P_f} \left( 1- \frac{1}{v_f(x)} \right) \] is less than 0, and a parabolic orbifold otherwise.

\section{\Teich space and Thurston iteration}
\label{sec:Teich}

Let $\T_f$ be the \Teich space modeled on the surface with marked points $(\Sphere,P_f)$ (see \cite[Section 2]{DH}) and $\M_f$  be the corresponding moduli space. 
 The space $\T_f$ can be defined as the quotient of the space of all diffeomorphisms from $(\Sphere,P_f)$ to the Riemann sphere $\P$ modulo post-composition with Moebius transformations and isotopies relative $P_f$. This is a $(p_f-3)$-dimensional complex manifold (in case $p_f \le 3$ it is a one point set). We write $\tau =\left\langle h \right\rangle$ if a point $\tau$ is represented by a diffeomorphism $h\colon (\Sphere,P_f) \to \P$. Correspondingly, points of $\M_f$ are represented by $h(P_f)$ modulo post-composition with Moebius transformations. Denote by $\pi:\T_f \to \M_f$ the canonical covering map which sends $h \mapsto h|_{P_f}$. The pure mapping class group of $(\Sphere,P_f)$ is canonically identified with the group of deck transformations of $\pi$. For more background on \Teich spaces see, for example, \cite{IT,H}.


Suppose $R$ is a  Riemann surface with marked points, denote by $P$ the set of marked points in $R$. If $R \sm P$ is hyperbolic,   we endow $R$ by default with the Poincar\'e metric of $R\sm P$ with constant curvature $-1$ .
 
The cotangent space at a point $\tau=\left\langle h \right\rangle$ in the \Teich space $T_f$ is canonically isomorphic to the space of all meromorphic quadratic differentials $Q(\P,h(P_f))$ on the Riemann surface with marked points corresponding to $\tau$ (which is $(\P,h(P_f))\;$) that are holomorphic on $\P \sm h(P_f)$ and have at worst simple poles in $h(P_f)$. The \Teich and Weil-Petersson norms for  $q \in Q(\P,h(P_f))$ are defined as follows:
$$ \|q\|_T = 2 \int_\P |q|$$ and 
$$ \|q\|_{WP} = \left( \int_\P \rho^{-2}|q|^2\right)^{1/2},$$
where $\rho$ is the hyperbolic distance element of $(\P,h(P_f))$. The duals of these two norms define Finsler metrics on $\T_f$ (the metric defined by the Weil-Petersson norm is not only Finsler but Hermitian). We write $d_T(\cdot,\cdot)$ and $d_{WP}(\cdot,\cdot)$ for the distances between points in the \Teich space with respect to the corresponding metric.

Consider an essential simple closed curve $\gamma$ in $(\Sphere,P_f)$. For each complex structure $\tau$ on $(\Sphere,P_f)$, there exists a unique geodesic $\gamma_\tau$ on $(\Sphere,P_f)$ in the homotopy class of $\gamma$ relative $P_f$.  We denote by $l(\gamma,\tau)$ the length of the geodesic $\gamma_\tau$ homotopic to $\gamma$ on the Riemann surface corresponding to $\tau \in \T_f$. This defines a continuous function from $\T_f$ to $\R_+$ for any given $\gamma$. Moreover, $\log l(\gamma,\tau)$ is a Lipschitz function with Lipschitz constant $1$ with respect to the \Teich metric (see \cite[Theorem 7.6.4]{H}; note that in \cite[Proposition 7.2]{DH} the constant is 2 because of a different normalization of the \Teich metric). We will use the same notation $l(\gamma, R)$ for the hyperbolic length of the geodesic homotopic to a  curve $\gamma$ in a hyperbolic surface $R$.

The length of a simple closed geodesic $\gamma$ on a hyperbolic Riemann surface $R$ is closely related to the supremum $M(\gamma,R)$ of moduli of all annuli on $R$ that are homotopic to this geodesic. The Collaring Lemma  \cite[Proposition~6.1]{DH} provides one estimate: $M(\gamma,R) > \frac{\pi}{l}-1$. The estimate $M(\gamma,R) \le \frac{\pi}{l}$ is trivial since the core curve of an annulus of modulus $m$ has length at most $\frac{\pi}{m}$.

 We will measure distances with respect to both the \Teich and the Weil-Petersson metrics in order to prove Theorem~\ref{thm:KevinConjecture}. For this we will need the following estimates \cite[Proposition 2.4 and Theorem 4.4]{McM}, which state that away from the boundary of the \Teich space, the two metrics are comparable. 

\begin{proposition}
\label{prop:metriccompairson}
\begin{itemize}
	\item[i.] There exists a constant $C^0$ such that for any tangent vector $v$ to $\T_f$ we have
		$$\|v\|_{WP} \le C^0 \|v\|_T.$$
	\item[ii.] If $l(\gamma,\tau)>\eps$ for all essential simple closed curves $\gamma$ then for any tangent vector $v$ to $\T_f$ at $\tau$ we have
		$$\|v\|_{WP} \ge C^1(\eps) \|v\|_T$$	
		where $C^1(\eps)$ is a constant depending on $\eps$.
\end{itemize}
\end{proposition}

We define the key player in the proof of Theorem~\ref{thm:Thurston} --- the Thurston pullback $\sigma_f$ --- as follows. Suppose $\tau \in \T_f$ is represented by a homeomorphism $h_\tau$. Consider the following diagram:

\begin{equation}
\begin{diagram}
\node{(\Sphere,P_f)}  \arrow{s,l}{f} 
\\
\node{(\Sphere,P_f)} \arrow{e,t}{h_\tau} \node{(\P,h_\tau(P_f))}
\end{diagram}
\label{dg1}
\end{equation}

We can pull back the standard complex structure $\mu_0$ on $\P$ to an almost-complex structure $f^* h_\tau^* \mu_0$ on $(\Sphere,P_f)$. By the Measurable Riemann mapping theorem  \cite{AB}, it induces a complex structure on $(\Sphere,P_f)$. Let $h_1$ be a conformal isomorphism between $(\Sphere,P_f)$ endowed with the complex structure $f^* h_\tau^* \mu_0$ and $\P$. Set $\sigma_f(\tau)=\tau_1$ where $\tau_1$ is the point represented by $h_1$.

Now we can complete the previous diagram by setting $f_\tau=h_\tau \circ f \circ h_1^{-1}$ so that it commutes:

\begin{equation}
\begin{diagram}
\node{(\Sphere,P_f)} \arrow{e,t}{h_1} \arrow{s,l}{f} \node{(\P,h_1(P_f))}
\arrow{s,l}{f_\tau}
\\
\node{(\Sphere,P_f)} \arrow{e,t}{h_\tau} \node{(\P,h_\tau(P_f))}
\end{diagram}
\label{dg2}
\end{equation}

Note that from definition of $f_\tau$, it follows that $f_\tau$ respects the standard complex structure $\mu_0$ and, hence, is rational. When we choose a representing homeomorphism $h_\tau$, we have the freedom to post-compose $h_\tau$ with any Moebius transformation; similairly, the choice of $h_\tau$ defines $h_1$ up to a post-composition by Moebius transformation. Thus, $f_\tau$ is defined up to pre- and post-composition by Moebius transformations.

It has been shown in \cite{DH} that $\sigma_f$ is a holomorphic self-map of $\T_f$ and that the co-derivative of $\sigma_f$  satisfies $(d\sigma_f(\tau))^*=(f_\tau)_*$ where $(f_\tau)_*$ is the push-forward operator on quadratic differentials. It is straightforward to prove $\|d\sigma_f\|_T=\|(d\sigma_f)^*\|_T \le 1$, and with a little more effort one gets $\|(d\sigma_f^k)\|_T < 1$, for some $k\in \N$,  when $f$ has hyperbolic orbifold (see \cite{TanLei}),
hence, $\sigma_f$ is weakly contracting on $\T_f$ with respect to the \Teich metric for any such $f$. Since $\T_f$ is path-connected, it follows that $\sigma_f$ has at most one fixed point if $f$ has hyperbolic orbifold, and every forward orbit of $\sigma_f$ converges to the fixed point in the case there exists one.

The following proposition \cite[Proposition 2.3]{DH} relates dynamical properties of $\sigma_f$ to the original question.

\begin{proposition}
\label{prop:fixedpts}
A Thurston map $f$ is equivalent to a rational function if and only if $\sigma_f$ has a fixed point.
\end{proposition}

Moreover, it is not hard to see that non-conjugate rational functions that are Thurston equivalent to $f$ would correspond to  different fixed points of $\sigma_f$. Therefore, the uniqueness part of Theorem~\ref{thm:Thurston} is clear.

The \emph{canonical} obstruction $\Gamma_f$ is the set of all homotopy classes of curves $\gamma$ that satisfy $l(\gamma,\sigma_f^n(\tau)) \to 0$ for all (or, equivalently, for some) $\tau \in \T_f$. The following theorems are due to Kevin Pilgrim \cite{P}. We give alternative proofs of these statements below.
\begin{theorem}[Canonical Obstruction Theorem]
\label{thm:CanonicalObstruction}
\begin{itemize}
\item[(1)] If for a Thurston map with hyperbolic orbifold its canonical obstruction is empty then it is Thurston equivalent to a rational function.
\item[(2)]  If the canonical obstruction is not empty then it is a simple Thurston obstruction.
\end{itemize}
\end{theorem}

\begin{theorem}[Curves Degenerate or Stay Bounded]
\label{thm:curvebound}
For any point $\tau \in \T_f$ there exists a bound $L=L(\tau, f)>0$ such that for any essential simple closed curve $\gamma \notin \Gamma_f$ the inequality $l(\gamma,\sigma_f^n(\tau)) \ge L$ holds for all $n$.  
\end{theorem}

\begin{remark} In the terms to be defined in the next section, the previous two theorems can be reformulated as follows: the sequence $\{\sigma_f^n(\tau)\}$ tends to $\S_{\Gamma_f}$ (Theorem~\ref{thm:CanonicalObstruction}) and the accumulation set of  $\pi(\{\sigma_f^n(\tau)\})$ in the compactified moduli space is a compact subset of  $\S_{[\Gamma_f]}$ (Theorem~\ref{thm:curvebound}).
\end{remark}

\section{The augmented \Teich space}

\label{sec:Aug}

Very relevant and good surveys on augmented \Teich spaces can be found in \cite{W1,W2}. Here we remind the reader of the basics that we will need later on.

Let $S$ be a topological surface of finite type (i.e., a surface of genus $g$ with $n$ punctures) with $m$ marked points   such that the Euler characteristic $\chi(S)=2-2g-n-m$ is negative. Recall that the  \Teich space of $S$ is the space of all Riemann surfaces with $m$ marked points of the same type as $S$. Each point in $\T(S)$ can be represented by a homeomorphism between $S$ and a Riemann surface. One defines the augmented \Teich space $\AT(S)$ as the space of all stable  Riemann surfaces with marked points \emph{with nodes} of the same type as $S$. The \emph{type} of a noded surface is defined by its topological type (more precisely, by the topological type  of a  surface one obtains by opening up all nodes) and the number of marked points (excluding nodes). In our case, the case of the sphere with marked points, any \emph{surface with nodes}  and marked points $R$ (or a \emph{noded surface}) is a collection of \emph{components} that are topological spheres with marked points, so that two components intersect in at most one marked point, each marked point belongs to at most two components, and the union of all components is connected and simply connected. The marked points that belong to two components of $R$ are called \emph{nodes}. Any component of a noded  surface $R$ can be obtained from a connected component of the complement of nodes in $R$ by adding to it all incident nodes as marked points. The genus, in this setting, is always $0$ thus the type of such a noded surface is determined by the number of marked points that are not nodes. \emph{Stable} noded surfaces are those for which every component is hyperbolic. We represent points in $\AT(S)$ not only by homeomorphisms but also by continuous maps from $S$ to a noded Riemann surface that are allowed to send a whole simple closed curve (or, which is the same up to homotopy, a closed annulus) in the complement of marked points  to a node. In other words, we allow to pinch some of the curves (or closed annuli) on $S$ into nodes. Abusing terminology, we will also call these curves on $S$ nodes.
The same idea is used to construct $\AM(S)$ --- the Bers compactification of the moduli space \cite{B2,B1}. The canonical projection from $\T(S)$ to $\M(S)$ extends to the canonical projection from $\AT(S)$ to $\AM(S)$. In other words, the following holds (see \cite{abi,B1}).

\begin{theorem} The quotient of $\AT(S)$ by the action of the pure mapping class group is compact.
\end{theorem}

The augmented \Teich space $\AT_f$ is a stratified space with strata corresponding to multicurves on $(\Sphere,P_f)$. We denote by $\S_\Gamma$ the stratum corresponding to the multicurve $\Gamma$, i.e., the set of all noded surfaces for which the nodes come from pinching all elements of $\Gamma$ and there are no other nodes. In particular, $\T_f = \S_\emptyset$. Strata of $\AM_f$ are labeled by equivalence classes $[\Gamma]$ of multicurves, where two multicurves $\Gamma_1$ and $\Gamma_2$ are in the same class if and only if one can be transformed to the other by an element of the pure mapping class group or, equivalently, if the respective elements of $\Gamma_1$ and $\Gamma_2$ separate points of $P_f$ in the same way. We naturally denote $\partial \T_f = (\AT_f \sm \T_f)$.

Given a multicurve $\Gamma$, each point in the stratum $\S_\Gamma$ is a collection of complex structures on the components of the corresponding topological noded surface with marked points. Therefore, $\S_\Gamma$ is the product of \Teich spaces of these components. We will refer to the points in the \Teich spaces of components as \emph{coordinates} of a point in $\S_\Gamma$. Within each stratum one can define its own natural \Teich (as the $\infty$-product of \Teich metrics of components) or Weil-Petersson (as the $2$-product of Weil-Petersson merics of components) metrics. The following theorem \cite{Masur} shows the interplay between the notion of the augmented \Teich space and the Weil-Petersson metric.

\begin{theorem}
\label{thm:masur}
 The augmented \Teich space is homeomorphic to the completion of the \Teich space with respect to the Weil-Petersson metric. Moreover, the restriction of the completed Weil-Petersson metric to each stratum is the Weil-Petersson metric of this stratum.
\end{theorem}

The following estimate on the Weil-Petersson norm of the coderivative of $\sigma_f$ shows that $\sigma_f$ is Lipschitz with respect to the Weil-Petersson metric on $\T_f$ and hence extends to its completion $\AT_f$. Combined with Theorem~\ref{thm:masur}, this  proves Theorem~\ref{thm:extensionA}.

\begin{proposition}
 $\|(d\sigma_f)^*\|_{WP} \le \sqrt{d_f}.$
\label{prop:Lipschitz}
\end{proposition}

\begin{proof}
Let us prove the statement for an arbitrary point $\tau=\langle h \rangle \in \T_f$. To simplify notation, set $g=f_\tau, d=d_f$ and $P=h(P_f)$. We need to prove then that $\|g_*q\|_{WP} \le \sqrt{d} \|q\|_{WP}$ for any $q \in Q(\P,P')$ where $P'=h_1(P_f)$ is the image of $P_f$ for $\sigma_f(\tau)=\langle h_1 \rangle$ (see the commutative \ref{dg2}).

If we take some small domain $U \subset (\P \sm P)$ with local
coordinate $\zeta$ such that it has exactly $d$ disjoint preimages
$U_i, i=\overline{1,d}$, with $g_i \colon I=U \to U_i$ the local branches
of $g^{-1}$, then 
\[ g_* q |_U = \sum_i g_i^* q. \]

Let $\rho^2$ and $\rho_1^2$ stand for the hyperbolic area elements
on $\P \sm P$ and $\P \sm P'$. The hyperbolic area element on $\P
\sm g^{-1}(P)$ is given by $g^* \rho^2$. The inclusion map $I
\colon \P \sm g^{-1}(P) \to \P \sm P'$ is length-decreasing;
therefore $\rho_1^2 \le g^* \rho^2$.

Now we can locally estimate:
\[ \int_U \frac{|g_* q|^2}{\rho^2} =
   \int_U \frac{|\sum_i g_i^* q|^2}{\rho^2} \le
   d \sum_i \int_{U} \frac{|g_i^*q|^2}{\rho^2}=
   d \sum_i \int_{U_i} \frac{|q|^2}{g^*\rho^2}
   \le  d \int_{g^{-1}(U)} \frac{|q|^2}{\rho_1^2},
\]
 where the first inequality follows from the fact that
 $$\left|\sum_{i=1}^d a_i \right|^2 \le d \sum_{i=1}^d |a_i|^2. $$

Combining local estimates, we get
\[ \|g_* q\|_{WP}  \le \sqrt{d} \|q\|_{WP},\]
as required. 
\qed\end{proof}

In Section~\ref{sec:extension} we refine the statement of Theorem~\ref{thm:extensionA} by showing that  a certain  extension of $\sigma_f$ to the boundary of $\AT_f$, defined in terms of pullbacks of complex structures on noded Riemann surfaces,  is continuous (see Theorem~\ref{thm:Continuity}) and, hence, coincides with the extension given by metric completion.

Note that $l(\gamma,\tau)$ extends continuously to a function from $\AT_f$ to $[0,+\infty ]$. A curve of length zero corresponds to a node; a curve of infinite length has to pass through at least one node (it has positive intersection number with at least one curve of length zero). By definition, $\tau \in \S_\Gamma$ when $l(\gamma,\tau)=0$ for those and only those homotopy classes of curves $\gamma$ that belong to $\Gamma$.

We can also view $l$ as a map $l \colon \AT_f \to [ 0,+\infty ]^{\H}$ where $\H$ is the set of all homotopy classes of essential simple closed curves. We endow $[ 0,+\infty ]^{\H}$ with the product topology.

\begin{proposition}
\label{prop:embedding}
The map $l \colon \AT_f \to [ 0,+\infty ]^{\H}$ is an embedding.
\end{proposition}

\begin{proof}
We know that $l$ is continuous on $\AT_f$ and that $l(\T_f)$ is mapped  onto its image by a homeomorphism (compare Theorems~3.12 and 3.15 in \cite{IT}). Similarly we can see, using the product structure of the boundary strata, that each of them is mapped to its image by a homeomorphism. Evidently images of different strata are disjoint. Hence the map is injective. It is straightforward to check that its inverse is also continuous.
\qed\end{proof}

In other words, the homotopy classes of nodes and the lengths of all simple closed geodesics uniquely define a point in $\AT_f$ and the topology on $\AT_f$ can be defined using the topology of $[ 0,+\infty ]^{\H}$.

\section{Technical background}
First we prove the following technical propositions.

\begin{proposition}
Let $X$ be an open hyperbolic subset of the Riemann sphere and $p$ be an isolated puncture of $X$.
Take a nested sequence $\{U_n\}$ of closed neighborhoods of $p$. Denote by $\rho_X$ the hyperbolic distance element on X and $\rho_n$ the hyperbolic distance element on the set $X \sm U_n$. If \, $\bigcap U_n = \{p\}$ then $\{\rho_n(x)\}$ tends to $\rho_X(x)$ for any point $x \in X$. Moreover, the convergence is uniform on compact subsets of  $X$.

\label{prop:MetricEstimate}
\end{proposition}

\begin{proof}
The identity inclusion of $X \sm U_n$ into $X \sm U_m$ for $n<m$ is obviously holomorphic, hence length-decreasing; the same is true for inclusions of $X \sm U_n$ into $X$. Therefore, $\{\rho_n(x)\}$ is decreasing and the limit is greater or equal than $\rho_X(x)$. 

Without loss of generality we assume that $p=\infty$. 
We know that $\rho_X(x) = 2/ \sup |f'(0)| $ where $f$ runs through the set of all holomorphic maps from the unit disk $\D$ to $X$ that send $0$ to $x$ (this is the definition of the Kobayashi  metric which is well known to coincide with the hyperbolic metric for hyperbolic Riemann surfaces). Let $\eps>0$. Pick such a map $f : \D \to X$ so that $2  / |f'(0)|  < (1 + \eps)\rho_X$. Set $D_t=\{|z| \le t, z \in \C\}$. We notice that $f({D_{1-\eps}})$ is compact in $\C$, hence bounded. Then $f(D_{1-\eps}) \subset X \sm U_n$ for all $n$ large enough. Hence $f_\eps (z):= f((1-\eps)z)$ is a holomorphic map from $\D$ to $X \sm U_n$ with $f(0)=x$ and $f_\eps'(0)=(1-\eps)f'(0)$. Thus by definition $$\rho_n(x) \le \frac{2}{f_\eps'(0)} = \frac{2}{(1-\eps)f'(0)} \le \frac{1+\eps}{1-\eps}\rho_X(x). $$
This shows that $\rho_n \to \rho_X$. Uniform convergence on compact subsets of $X$ is evident. 
\qed\end{proof}


\begin{proposition}

Let  $\{U_n\}$ be an increasing nested sequence of open subsets of $\P$ such that the complement of the union $U$ of all $U_n$ consists of finitely many points 
and $\{f_n\colon\P\to\P\}$ be a sequence of continuous maps that fix three points $a_1,a_2,a_3$ on $\P$. If $f_n^{-1}$ is well-defined (i.e. every point has exactly one preimage) and conformal on $U_n$ for all $n$, then $\{f_n\}$ converges uniformly to the identity mapping on $\P$.

\label{prop:ParameterDependence}\end{proposition} 

\begin{proof}
Let $V_n=f_n^{-1}(U_n)$. Suppose first that the three points $a_1,a_2,a_3$ are inside $V_n$ for all $n$. Consider the sequence of conformal inverses $\{f^{-1}_n \colon U_n \to V_n\}$. All of these functions are conformal on $U_1 \sm \{a_1,a_2,a_3\}$ and do not assume values  $\{a_1,a_2,a_3\}$. Hence, by Montel's theorem, $\{f^{-1}_n\}$ forms a normal family on $U_1 \sm \{a_1,a_2,a_3\}$ and we can choose a subsequence $\{f^{-1}_{i(1,n)}\}$ converging locally uniformly on $U_1$ to a conformal map. By the same reasoning, we choose a subsequence $\{f^{-1}_{i(2,n)}\}$ of $\{f^{-1}_{i(1,n)}\}$ that converges locally uniformly on $U_{2}$ and so on. Then the diagonal subsequence $\{f^{-1}_{i(m,m)}\}$ converges locally uniformly on $ U$. Since the limit is conformal on $\P$ except for finitely many points, it is in fact conformal on the whole sphere, and since it fixes $a_1,a_2,a_3$ it is the identity. Note that the same reasoning implies that we can choose a subsequence converging locally uniformly to the identity from any subsequence of the sequence $\{f^{-1}_n\}$. Therefore $\{f^{-1}_n\}$ converges locally uniformly to the identity on $U$. Since every point in $U_n$ has a unique $f_n$-preimage, it is easy to see that $\{f_n\}$ converges uniformly to the identity  as required.

We now consider the general case when $a_1,a_2,a_3$ are chosen arbitrarily. Pick points $b_1,b_2,b_3$  in $U_1$ and set $c_i^n=f^{-1}_n(b_i)$ for $i=\overline{1,3}$ so that $c_i^n \in V_n$ for all $n$. Set $g_n=f_n \circ \phi_n $ where $\phi_n$ is the Moebius transformation that sends $b_i$ to $c_i^n$. By the earlier arguments, $\{g_n\}$ uniformly converges to the identity. In particular, $g_n(\phi_n^{-1}(a_i))=f_n(a_i)=a_i$ implies that $\phi_n^{-1}(a_i) \to a_i$ for $i=\overline{1,3}$. Therefore, $\{\phi_n^{-1}\}$ converges uniformly to the identity and the general case follows.
\qed\end{proof}

\section{Extension of $\sigma_f$ to the augmented \Teich space}

\label{sec:extension}
 
 Let $\tau =\langle {h_\tau} \rangle \in \S_\Gamma$ where $\Gamma \neq \emptyset$. Let $R$ be the Riemann surface with nodes corresponding to $\tau$. We look again at the  commutative diagram~\eqref{dg1}, only this time on the bottom-right we have not the Riemann sphere but  a Riemann surface $R$ with nodes and marked points that consists of several Riemann spheres touching at the nodes with the image of $P_f$ on it (see the commutative diagram~\eqref{dg5} below). Consider the full preimage $f^{-1}\circ {h_\tau}^{-1}(N)$ of the set of nodes $N$ of $R$ on $\Sphere$. We obtain a topological noded surface $T^0$ with nodes $f^{-1}\circ {h_\tau}^{-1}(N)$ which is not necessarily stable. For example, it is not stable if a component of $f^{-1}(\gamma)$ is non-essential for some node $\gamma$. We consider the stabilisation $T^1$ of $T^0$ which is defined as follows. Every non-hyperbolic component of $(T^0,P_f)$ is a sphere with at most two marked points and nodes, hence is obtained by pinching either 
\begin{itemize}
	\item[i.] a  simple closed curve that is non-essential in $\Sphere \sm P_f$
	\item[ii.] or a pair of  simple  closed curves that are homotopic to each other in $\Sphere \sm P_f$.
\end{itemize}
In both cases, these components have a unique possible complex structure, thus they do not carry any information, and we collapse each of them to a point to produce a stable noded surface $(T^1,P_f)$. In the first case, we obtain an ordinary point of $T^1$ if the pinched non-essential closed curve was null-homotopic and a marked point of $T^1$ otherwise. In the second case, we obtain a node of $T^1$ if the pinched curves are essential, otherwise we get an ordinary point or a marked point as in the first case. Note that with this construction several adjacent components might be pinched to a single point.

Denote by $T$ the topological noded surface model of $R$.  In other words, $T$ is $R$ viewed as a topological surface. Let $\id\colon R\to T$ be the canonical homemorphism between the two surfaces.
Let $\hh =  \id \circ  h_\tau$ be the canonical projection map $\hh \colon (\Sphere,P_f) \to T$; we see that $\hh$ sends any connected component of $ {h_\tau}^{-1}(N)$ to a point and maps any connected component of $\Sphere \sm {h_\tau}^{-1}(N)$  homeomorphically onto a component of $T$ without finitely many points. Similarly, let $\hh_1$ be the canonical projection map $\hh_1  \colon (\Sphere,P_f) \to T^1$ that sends any connected component of $f^{-1}\circ {h_\tau}^{-1}(N)$ to a point; any connected component of $\Sphere \sm f^{-1}\circ {h_\tau}^{-1}(N)$  is mapped by $\hh_1$  to a point, if that component is non-hyperbolic, otherwise it is homeomorphically mapped onto a component of $T^1$ without finitely many points.
 The maps $\hh$ and $\hh_1$ are evidently injective on $P_f$. 

Each component $C^1_i$ of $T^1$ is mapped to some component $C_j$ of $T$ by a branched cover $f^{C^1_i}$  that can be defined using the following diagram which commutes on each component of $T^1$.
 
$$
\begin{diagram}
\node{(\Sphere,P_f)} \arrow{e,t}{\hh_1} \arrow{s,l}{f} \node{(T^1,\hh_1(P_f))}
\arrow{s,l}{\{f^{C^1_i}\}}
\\
\node{(\Sphere,P_f)} \arrow{e,t}{{\hh}} \node{(T,{\hh}(P_f))}
\end{diagram}
$$
 
We define maps $h_\tau^{C_j}\colon C_j \to R_j$, where $R_j$ are the corresponding components of $R$, to be the unique maps that make the following diagram commute:

$$
\begin{diagram}
\node[2]{(T,{\hh}(P_f))} \arrow{se,t}{\{h_\tau^{C_j}\}} 
\\
\node{(\Sphere,P_f)} \arrow{ne,t}{{\hh}}  \arrow[2]{e,t}{{h_\tau}} \node[2]{(R,{h_\tau}(P_f))}
\end{diagram}
$$


 We can now define a complex structure on $T^1$ component-wise by pulling back  complex structures of corresponding components of $R$ by $\{f^{C^1_i}\}$ in the same manner as we did when we defined $\sigma_f$ on $\T_f$. 
For every component $C^1_i$ we construct a commutative diagram analogous to the commutative diagram~(\ref{dg2})
\begin{equation}
\begin{diagram}
\node{C^1_i} \arrow{e,t}{h^{C^1_i}_1} \arrow{s,l}{f^{C^1_i}} \node{R^1_i}
\arrow{s,r}{f^{C^1_i}_\tau}
\\
\node{C_j} \arrow{e,t}{h^{C_j}_\tau} \node{R_j}
\end{diagram}
\label{dg3}
\end{equation}
 Note that all components on the left are topological spheres and all components on the right are Riemann spheres. Thus, the situation is exactly as above with one important exception: $f^{C^1_i}$ is not a self-map but a map between two different topological spheres.

The topological noded surface $T^1$ is now endowed with complex structure and becomes a Riemann surface with nodes $R^1$ of the same type as $R$. Let $h_1 \colon T^1 \to R^1$ be the map that acts on every component $C^1_i$ as $h_1^{C^1_i}$. The following diagram summarizes  the aforesaid.

\begin{equation}
\label{dg5}
\begin{diagram}
\node{(\Sphere,P_f)} \arrow{e,t}{\hh_1} \arrow[2]{s,l}{f} \node{(T^1,\hh_1(P_f))} \arrow{s,l}{\{f^{C^1_i}\}}  \arrow{e,t}{h_1} \node{(R^1,h_1\circ\hh_1(P_f))}  \arrow[2]{s,r}{\{f_\tau^{C^1_i}\}}
\\
\node[2]{(T,\hh(P_f))} \arrow{se,t}{\{h^{C_j}_\tau\}}
\\
\node{(\Sphere,P_f)} \arrow{ne,t}{\hh} \arrow[2]{e,t}{{h_\tau}} \node[2]{(R,{h_\tau}(P_f))}
\end{diagram}
\end{equation}

We  set $\sigma_f(\tau)=\langle \hh_1\circ h_1\rangle$. It is straightforward to check that $\sigma_f$ is now well-defined as a self-map of $\AT_f$. 


Note that this  definition of  $\sigma_f$ on the boundary of the augmented \Teich space immediately implies the following.

\begin{proposition}
\label{prop:invariantstrata}
  A stratum $\S_\Gamma$ is mapped by $\sigma_f$ into the stratum $\S_{f^{-1}(\Gamma)}$. In particular, $\sigma_f$-invariant boundary strata are in one-to-one correspondence with completely invariant multicurves.
\end{proposition}

Fix a  component $C':=C^1_i$ of the noded surface $T^1$ and the corresponding cover $f^{C'}$ that sends $C'$ to a component $C:=C_j$ of the noded surface $T$. Select points $a,b$ and $c$ from $P_f$ in such a way that no two points from $\{a,b,c\}$ are separated from $C$ by any single curve from $\Gamma$ (we say that $a,b,c$ \emph{single out} the component $C$). Obviously, any  triple of points (that are not on nodes) on a noded surface of genus 0 singles out exactly one of its components. Similarly, choose (possibly different) points $a',b'$ and $c'$ in $P_f$ that single out $C'$ in $T^1$. For all points $\tau$ in $\T_f$, we normalize the homeomorphisms $h_\tau$ and $h_1$ in the commutative diagram~(\ref{dg2}) so that  $h_\tau(a)=h_1(a')=0$ and $h_\tau(b)=h_1(b')=1$ and $h_\tau(c)=h_1(c')=\infty$ (or any other selected values). Since $f_{\tau}$ is defined up to pre- and post-compositions with Moebius transformations, fixing these normalization conditions defines all $f_{\tau}$ uniquely. 

Let $p^C$ be the naturally defined projection from $\Sphere$ to $C$ that sends connected components of the complement of $C$ to the nodes that separate these components from $C$ and $P_f^C:=p^C(P_f)$ be the set of nodes and marked points on $C$; define $p^{C'}$ and $P_f^{C'}$ in the same manner. Then $a$, $b$ and $c$ single out $C$ if and only if $p^C$ is injective on $\{a,b,c\}$.
  For any point $\tau \in \S_\Gamma$, we define $f^{C'}_{\tau}$ uniquely by imposing the same normalization on the functions in the commutative diagram~(\ref{dg3}): $h^C_\tau(p^C(a)) = h^{C'}_1(p^{C'}(a')) = 0$ and $h^C_\tau(p^C(b))=h^{C'}_1(p^{C'}(b'))=1$ and $h^C_\tau(p^C(c))=h^{C'}_1(p^{C'}(c'))=\infty$.

 \begin{proposition}
  Let $\{\tau_n\} \in \T_f$ be a sequence converging to a point $\tau \in \S_\Gamma$. With the normalizations as above $\{f_{\tau_n}\}$ converges uniformly to $f^{C'}_{\tau}$ on any compact set in the complement of the $h_1^{C'}$-image of the nodes.
   \label{prop:UniformConvergence}
 \end{proposition}
 
 \begin{proof}
Given representing maps $h_\tau$, $h_{\tau_n}$ and  $h_1$, $h_{1,\tau_n}$, set $p^C_n=h_\tau^C\circ p^C \circ h_{\tau_n}^{-1}$ and $p^{C'}_n=h_1^{C'}\circ p^{C'} \circ h_{1,\tau_n}^{-1}$ to complete the following commutative cube, where the front side  of the cube is the commutative diagram~\eqref{dg2}, the back side of the cube is the commutative diagram~\eqref{dg3}, and the dotted arrows are the corresponding projection maps.
\begin{equation}
\begin{diagram}
 \node[2]{(C',P^{C'}_f)} \arrow[2]{e,t}{h^{C'}_1}  \arrow[2]{s,r,1}{f^{C'}} \node[2]{(\P, h_{1}^{C'}(P^{C'}_f))} \arrow[2]{s,r}{f^{C'}_\tau}
\\ \node{(\Sphere,P_f)}  \arrow[2]{e,t,1}{h_{1,\tau_n}}  \arrow{ne,t,..}{p^{C'}} \arrow[2]{s,r}{f} \node[2]{(\P, h_{1,\tau_n}(P_f))} \arrow[2]{s,r,1}{f_{\tau_n}} \arrow[1]{ne,t,..}{p^{C'}_n}
\\  \node[2]{(C,P^{C}_f)}  \arrow[2]{e,t,1}{h^{C}_\tau}  \node[2]{(\P, h^C_{\tau}(P^C_f))}
\\ \node[1]{(\Sphere,P_f)} \arrow[2]{e,t}{h_{\tau_n}} \arrow{ne,t,..}{p^{C}} \node[2]{(\P, h_{\tau_n}(P_f))} \arrow[1]{ne,t,..}{p^{C}_n}
\end{diagram}
\label{dgcube}
\end{equation}

For any point $q \in P_f$, we have $h_{\tau_n}(q) \to h^C_\tau(p^C(q))$. Indeed, for points $a,b,c$ this is by definition, for the rest of the marked points it follows from the fact that the cross ratios of marked points are continuous functions on $\AT_f$. By assumption, $l(\gamma,\tau_n) \to 0$ for all $\gamma \in \Gamma$, so with the chosen normalizations, the hyperbolic geodesics $\gamma_{\tau_n}$  on $\tau_n$ are contained in arbitrarily small spherical neighborhoods of the $h^C_\tau$ image of the corresponding nodes as $n$ goes to infinity.  Thus, since $h_{\tau_n}$ and $h_{\tau}$ are defined up to homotopy relative $P_f$, we can assume that $p_n^C$ is conformal onto $\P \sm U$ where $U$ is  a small neighborhood of  $ h^C_{\tau}(P^C_f)$.  Following the commutative diagram~(\ref{dgcube}), we see that $p_n^{C'}$ is conformal onto $\P \sm (f_\tau^{C'})^{-1}(U)$.
%
 %

 Proposition~\ref{prop:ParameterDependence} implies that both $\{p_n^{C}\}$ and $\{p_n^{C'}\}$ uniformly converge to the identity  which gives us the desired result.
 \qed\end{proof}

\begin{remark} Note that if we normalize our sequence in the same way as above, but choosing  $\{a,b,c\}$ and $\{a',b',c'\}$ so that these triples do not satisfy the conditions described above (which means that $\{a',b',c'\}$ singles out a component $C'$ of $\sigma_f(\tau)$ that does not map to the component $C$ of $\tau$ singled out by $\{a,b,c\}$), then $f_{\tau_n}$ converges uniformly on any compact set in the complement of the image of the nodes to a constant map. Indeed, with the chosen parametrization the component $C'$ becomes larger and larger as $n$ grows and is mapped into the complement of the component $C$ and this complement becomes smaller and smaller. 
\end{remark}

The last proposition is not only a useful tool for proving the next theorem but is also of interest in its own right. We see that under the normalization assumptions given above, as $\tau$ tends to a boundary point $\tau_0$ in $\AT_F$, the map $f_\tau$ deforms in a continuous fashion with respect to the locally uniform convergence in the complement of the nodes. All possible limits are either constant maps or rational  maps of possibly smaller degree that map components of $\sigma_f(\tau_0)$ to the corresponding components of $\tau_0$.

\begin{theorem}
  The map $\sigma_f$ as defined above is continuous on $\AT_f$.
  \label{thm:Continuity}
\end{theorem}

\begin{proof}

We see that this map by definition preserves the product structure on every stratum in the following sense. If $\S_\Gamma \cong \T_1 \times \T_2 \times \ldots \times \T_n$ then for each point $\tau = (\tau^1, \tau^2, \ldots, \tau^n) \in \S_\Gamma$  we have 

$$\sigma_f(\tau)=(\sigma_1(\tau^{i_1}), \sigma_2(\tau^{i_2}) , \ldots, \sigma_m(\tau^{i_m}))\in \S_{f^{-1}(\Gamma)}$$
where $\sigma_1, \ldots, \sigma_m$ are pullback maps for the covers $f_1, \ldots, f_m$ from components of $\tau_1$ to components of $\tau$. It follows immediately that $\sigma_f$ is continuous on each stratum of $\T_f$. Since every stratum lies on the boundary of finitely many strata, 
it suffices to show sequential continuity for sequences that lie within a single stratum converging to a boundary point. Moreover, we can assume that the stratum is $\T_f$ itself; the other cases follow if we apply the same argument to $\sigma_1, \ldots, \sigma_m$.

We are going to show now that for any $\{\tau_n \}\in \T_f$ such that $\tau_n \to \tau \in \partial \T_f$, we have $l(\gamma', \sigma_f(\tau_n)) \to l(\gamma', \sigma_f(\tau))$ for every $\gamma'$, which will conclude the proof of the theorem by Proposition~\ref{prop:embedding}. Denote $\sigma_f(\tau_n)=\tau'_n$ and $\sigma_f(\tau)=\tau'$ to simplify the notation. We consider three cases: when $\gamma'$ is a node of $\tau'$, when $\gamma'$ intersects at least one of the nodes, and the rest of the homotopy classes of simple closed curves.

 First, let us look at the case when $\gamma'$ is a node of $\tau'$ (i.e. $l(\gamma', \tau')=0$). Then, by definition, $\gamma'$ is homotopic to at least one preimage of a curve $\gamma$ which is a node of $\tau$. Then $f$ maps this preimage onto $\gamma$ as a cover of some certain degree, say $d$.  The corresponding preimage $\delta$ of the geodesic homotopic to $\gamma$ in $\tau_n$ has length equal to $d \cdot l(\gamma,\tau_n)$ with respect to the Poincar\'e metric on $\P\sm f_{\tau_n}^{-1}(P_f)$. Since filling in some of the punctures decreases Poincar\'e metric we get that $l(\gamma',\tau'_n) \le l(\delta,\tau'_n) \le d \, l(\gamma,\tau_n)$. Since $l(\gamma,\tau_n) \to l(\gamma,\tau)=0$ we conclude that $l(\gamma',\tau'_n) \to 0 =l(\gamma',\tau')$.
 
 The second case is when $l(\gamma',\tau')= \infty$. In this case $\gamma'$ must have positive intersection number with at least one node $\delta$ of $\tau'$. We already know that $l(\delta,\tau'_n)$ tends to 0. By the Collaring Lemma, it follows that $l(\gamma',\tau'_n)$ tends to infinity.
 

We conclude the proof 
by showing that $l(\gamma',\tau_n') \to l(\gamma',\tau')$ for the rest of the curves (i.e., when $l(\gamma',\tau') \notin \{0,\infty\}$). Let $C'$ be the component  of $\tau'$ that contains $\gamma'$ and $C$ be the corresponding component of $\tau$. Fix normalization conditions for representing homeomorphisms $h_\tau, h_1$ and $h_{\tau_n}, h_{1,\tau_n}$ as above so that  Proposition~\ref{prop:UniformConvergence} applies. 
 
 
Denote by $P_{f,n}$ the $h_{1,\tau_n}$-image of $P_f$; let $\rho_n$ be the hyperbolic distance element  on $\P \sm P_{f,n}$ and $\rho$  be the hyperbolic distance element  on $\P \sm h^{C'}_1(P^{C'}_f)$.  Define $\rho_n^1$ to be the hyperbolic distance element on $\P \sm U_n$, where $U_n$ is a small neighborhood of  $h^{C'}_1(P^{C'}_f)$ that contains $P_{f,n}$. Recall that $p^{C'}_n=h_1^{C'}\circ p^C \circ h_{1,\tau_n}^{-1}$ converges uniformly to the identity  (see the proof of Proposition~\ref{prop:UniformConvergence}) which implies that $\{U_n\}$ can be chosen so that the intersection thereof is $h^{C'}_1(P^{C'}_f)$. Define $\rho_n^2$ to be the hyperbolic distance element on $\P \sm P'_{f,n}$ where $ P'_{f,n} \subset  P_{f,n}$ is such that $P'_{f,n}$ has exactly one point in every connected component of $U_n$. Then, by the Schwarz lemma, we have  $\rho_n^1 \ge  \rho_n \ge \rho_n^2$. 

Applying Proposition~\ref{prop:MetricEstimate}, we get that point-wise $  \rho^1_n \to \rho$. On the other hand,  we clearly have $\rho_n^2 \to \rho$ because $\{(\P,P'_{f.n})\}$ is  a converging sequence in the \Teich  space of $(\Sphere,P^{C'}_f)$.
This shows that $\rho_n \to \rho$ point-wise.

Note that geodesics on all $\P \sm P_{f,n}$ in the same homotopy class as $\gamma'$ live in a compact subset of $\P \sm h^{C'}_1(P^{C'}_f)$ by the Collaring Lemma. But on any compact set, the convergence of the hyperbolic length elements will be uniform. One easily deduces $l(\gamma',\tau_n') \to l(\gamma',\tau')$. 
Since $\gamma'$ was chosen arbitrarily, we conclude that $l(\gamma',\tau_n') \to l(\gamma',\tau')$ for all homotopy classes $\gamma'$  of simple closed curves in $(\Sphere,P_f)$. 
\qed\end{proof}

\section{Classification of invariant boundary strata}

\label{sec:strata}

As was mentioned above, every invariant boundary stratum corresponds to a completely invariant multicurve  $\Gamma$. We want to classify the topological behavior of $\sigma_f$ near invariant boundary strata $\S_\Gamma$ according to the value of $\lambda_\Gamma$. An invariant stratum $\S$ of $\AT_f$ will be called \emph{weakly attracting} if there exists a nested
decreasing sequence of neighborhoods $U_n$ such that
$\sigma_f(U_n) \subset U_n$ and $\bigcap U_n= {\S}$. An
invariant stratum $\S$ of $\AT_f$ will be called \emph{weakly
repelling} if for any compact set $K \subset \S$ there exists a
neighborhood $U \supset K$ such that every point of $U \cap \T_f$ escapes
from $U$ after finitely many iterations (for every $\tau \in U
\cap \T_f$, there exists an $n \in \N$ such that $\sigma_f^{n}(\tau) \notin U$).

\begin{proposition}
\label{prop:attracting}
   If $\Gamma=\{\gamma_1, \gamma_2, \ldots, \gamma_m \}$ is a simple obstruction, then $\S_\Gamma$ is weakly attracting.
\end{proposition}
\begin{proof}
By Proposition~\ref{prop:positive}, we can choose a vector $v>0$ such that $M_\Gamma v \ge v$. Consider $V_n \subset \AT_f$ for $n\in\N$, the set of all points $\tau= \left\langle h \right\rangle$
   of the augmented \Teich space for which there exist mutually disjoint annuli
   $A_i$ homotopic to $\gamma_i$ in $(\Sphere, P_f)$ such that $\mod h(A_i) > n v_i$ for $i=\overline{1,m}$. 
   
   Construct disjoint annuli $B_i$ for $i=\overline{1,m}$ in $(\Sphere, P_f)$ that contain the union of all components $A_{i,j,k}$ of the preimage of $A_j$ that are homotopic to $\gamma_i$. Pick a diffeomorphism $h_1$ that represents $\sigma_f(\tau)$ (i.e. $\sigma_f(\tau)=\langle h_1 \rangle$). By definition, $f_\tau =h \circ f \circ h_1^{-1}$ is holomorphic and non-ramified on all $h_1(A_{i,j,k})$. Therefore, $$\mod h_1(A_{i,j,k}) = (\deg f|_{A_{i,j,k}}\colon A_{i,j,k}\to A_j)^{-1} \mod h(A_j).$$
   It follows from the Gr\"otzsch inequality that 
   $$\mod h_1(B_i) \ge \sum_{A_{i,j,k}} \mod h_1(A_{i,j,k}) =  \sum_{A_{i,j,k}} (\deg f|_{A_{i,j,k}}\colon A_{i,j,k}\to A_j)^{-1} \mod h(A_j).$$
   If we write this inequality in vector form we get simply $$\mod h_1(B_i) \ge M_\Gamma \mod h(A_i) > M_\Gamma (nv) 
   \ge nv.$$
   Thus $\sigma_f(\tau) \in V_n$ for all $\tau \in V_n$.
   
   As mentioned above, large annuli in homotopy classes of $\gamma_i$ exist if
    and only if the lengths $l(\gamma_i,\tau)$ are short. It follows that  $\bigcap V_n =
     \overline{\S_\Gamma}$. Denote by $U_n$ the intersections of $V_n$ with the union of all strata
      $\S_{\hat{\Gamma}}$  where $\hat{\Gamma}\subseteq \Gamma$.  Then clearly $\bigcap U_n =
       \S_\Gamma$, and $U_n$ are still invariant since $\hat{\Gamma}\subseteq \Gamma$ implies
        $f^{-1}(\hat{\Gamma})\subseteq f^{-1}(\Gamma)=\Gamma$.    
\qed\end{proof}

 We will need the following (see Proposition 1.2.2 in \cite{KH})
 
\begin{proposition}
\label{prop:norm} Let $M \in \C^{m \times m}$ be a matrix such that all
eigenvalues of $M$ have absolute value strictly less than $\delta$.
Then there exists a norm $\| \cdot \|$ on $\C^m$ such that the 
operator norm satisfies $\|M\| \le \delta$. 
\end{proposition}

Recall the following analytic tool from \cite[Theorem 7.1]{DH}.

\begin{proposition}
\label{prop:estimate}
Let $X$ be a Riemann surface and $P \subset X$ a finite set. Set $X'=X \sm P$, $p=\#P$, and choose $L<2\log(\sqrt{2}+1)$. Let $\gamma$ be a simple closed geodesic on $X$, and $\{\gamma'_1,\ldots,\gamma'_s\}$ be the closed geodesics of $X'$ homotopic to $\gamma$ in $X$ and of length $<L$. Then
$$ 1/l-2/\pi-(p+1)/L < \sum_{i=1}^s 1/l'_i < 1/l +2(p+1)/\pi.$$
\end{proposition}

 Define $Z(\Gamma,\tau )=(1/l(\tau,\gamma_1),\ldots,1/l(\tau,\gamma_m))^T$. Then $\|Z(\Gamma,\tau)\|$  can be roughly thought as the inverse of distance between $\tau$ and the boundary of $\AT_f$, i.e. the larger $\|Z(\Gamma,\tau)\|$ is, the closer $\tau$ is to some stratum $\S_{\hat{\Gamma}}$  with $\hat{\Gamma}\subseteq \Gamma$. 

\begin{proposition}
\label{prop:contraction}
Let  $\Gamma=\{\gamma_1, \gamma_2, \ldots, \gamma_m \}$ be an invariant multicurve with $\lambda_\Gamma <1$. Pick a norm $\|\cdot\|$ for $M_\Gamma$ on $\R^m$ as in Proposition~\ref{prop:norm} with $\lambda_\Gamma<\delta<1$.
Take $$U(\Gamma)={\left\{ \tau \in \T_f | \inf_{ \gamma \notin \Gamma} l(\gamma,\tau) \ge L  \right\}},$$
where $0<L<2\log(\sqrt{2}+1)$. Then for every $\delta' > \delta$ there exists  $T(L,\delta,\delta')>0$ such that $\|Z(\Gamma,\sigma_f(\tau))\|<\delta'\|Z(\Gamma,\tau)\|$ for all $\tau \in U(\Gamma)$ with $\|Z(\Gamma,\tau)\|>T(L,\delta,\delta')$.
\end{proposition}

\begin{proof}
Let $\tau$ and $\tau_1=\sigma_f(\tau)$ be represented by $h$ and $h_1$ respectively.
Let us apply Proposition~\ref{prop:estimate} to the surface $X=\P \sm h_1(P_f)$ and its finite subset $P=h_1(f^{-1}(P_f) \sm P_f)$. Every geodesic on $X \sm P$ is mapped by a non-ramified cover $f_\tau$ onto a geodesic of $\P \sm h(P_f)$. Therefore, those geodesics on $X \sm P$ that are not preimages of geodesics with homotopy classes in $\Gamma$ have lengths at least $L$. Denote as before by $\gamma_{i,j,k}$ preimages of $\gamma_j$ that are homotopic to $\gamma_i$ for all pairs $i,j$. Then $l(\gamma_{i,j,k},X)=(\deg f|_{\gamma_{i,j,k}}\colon \gamma_{i,j,k}\to \gamma_j) l(\gamma_j,\tau)$.
 Also note that $\# (f^{-1}(P_f) \sm P_f) \le d_f \#P_f=d_f p_f$. Thus, we get for all $i=\overline{1,m}$:
		$$1/(l(\gamma_i,\tau_1))<2/\pi+(dp_f+1)/L+\sum_{\gamma_{i,j,k}} (\deg f|_{\gamma_{i,j,k}}\colon \gamma_{i,j,k}\to \gamma_j)^{-1} l(\gamma_j,\tau)^{-1}.$$
		Expressing these inequalities in vector form, we get $Z(\Gamma,\tau_1)<M_\Gamma Z(\Gamma,\tau)+c$, where $c$ is a constant vector depending only on $L$ with $\|c\|=C(L)$. By Proposition~\ref{prop:norm} we get: $$\|Z(\Gamma,\tau_1)\|<\|M_\Gamma Z(\Gamma,\tau)+c\| \le \|M_\Gamma\|\, \|Z(\Gamma,\tau)\|+\|c\| < \delta \|Z(\Gamma,\tau)\|+C(L).$$ It is evidently sufficient to take $T(L,\delta,\delta')=C(L)/(\delta'-\delta)$.
\qed\end{proof}

We can extend the setting of the previous proposition by considering any completely  invariant multicurve $\Gamma$ which is not a simple obstruction. We write $$ M_\Gamma = \left( 
\begin{array}{cc}
	M_{11}  & 0 \\
	M_{21}  & M_{22}
\end{array}
 \right),
$$ where the leading eigenvalue $\lambda_1$ of $M_{11}$ is less than 1. Denote by $s$ the dimensions of $M_{11}$. As before, using Proposition~\ref{prop:norm}, define a norm on $\R^s$ such that $\|M_{11}\|\le \delta<\lambda_1$. Extend this norm to a semi-norm on $\R^m$ which depends only on the first $s$ coordinates. It is straightforward to check that the proof of  Proposition~\ref{prop:contraction} works in this setting. We conclude with the following proposition.

\begin{proposition}
\label{prop:seminorm}
Let   $\Gamma=\{\gamma_1, \gamma_2, \ldots, \gamma_m \}$ be a completely invariant multicurve which is not a simple obstruction. Fix $L\in \R$ such that $0<L<2\log(\sqrt{2}+1)$ and define $$U(\Gamma)={\left\{ \tau \in \T_f | \inf_{ \gamma \notin \Gamma} l(\gamma,\tau) \ge L  \right\}}.$$
 Then there exists a semi-norm $\|\cdot\|$ on $\R^m$ and two real numbers $T(L)>0$  and $0<\delta<1$ such that  $\|Z(\Gamma,\sigma_f(\tau))\|<\delta\|Z(\Gamma,\tau)\|$ for all $\tau \in U(\Gamma)$ with $\|Z(\Gamma,\tau)\|>T(L)$.
\end{proposition}

\begin{corollary} 
\label{cor:repelling} If $\Gamma=\{\gamma_1, \gamma_2, \ldots, \gamma_m \}$ is a completely invariant multicurve which is not a simple obstruction, then $\S_\Gamma$ is weakly repelling.
\end{corollary}

\begin{proof}
 For any compact $K \subset \S_\Gamma$ we have $\inf_{\tau \in K, \gamma \notin \Gamma} l(\gamma,\tau)=k>0$.   Choose $L=(1/2)\min\{k,2\log(\sqrt{2}+1)\}$. Consider $$U=\{ \tau \in \T_f | \inf_{ \gamma \notin \Gamma} l(\gamma,\tau) \ge L  \:\: \text{and} \:\: \|Z(\Gamma,\tau)\|>T\},$$  where $T=T(L)$ is as in Proposition~\ref{prop:seminorm}. Clearly $\overline{U} \supset K$. Suppose that there exists a point $\tau \in \T_f$ that does not escape from $U$, i.e., $\sigma_f^n(\tau) \in U$ for all $n$. Then, on one hand, we have $\|Z(\Gamma,\sigma_f^n(\tau))\|>T$ for all $n$ and, on the other hand, $\|Z(\Gamma,\sigma_f^{n+1}(\tau))\|<\delta\|Z(\Gamma,\sigma_f^n(\tau))\|$ which is  a contradiction.
\qed\end{proof}

\section{Proofs of Thurston's and Canonical Obstruction Theorems}
\label{sec:proofs}

In the previous section we described the behavior of $\sigma_f$ near invariant boundary strata. The understanding of the action of $\sigma_f$ near infinity plays a key role in our proof of Thurston's theorem.

The following proposition is essentially \cite[Lemma 5.2]{DH} (property \emph{iii.} is not stated in \cite{DH}, but follows from the proof given there).

\begin{proposition}
\label{prop:covers}
 There exists an intermediate cover $\M'_f$ of $\M_f$ (so that
  $\T_f \stackrel{\pi_1}{\longrightarrow} \M'_f \stackrel{\pi_2}{\longrightarrow} \M_f$ 
 are covers and $\pi_2\circ\pi_1=\pi$) such that
 \begin{itemize}
	\item[i.] $\pi_2$ is finite,
	\item[ii.]  the diagram
\begin{equation}
\label{diag:covers}
\begin{diagram}
\node{\T_f} \arrow[2]{e,t}{\sigma_f} \arrow{se,l}{\pi_1} \arrow[2]{s,l}{\pi} \node[2]{\T_f}
\arrow[2]{s,l}{\pi}
\\
\node[2]{\M'_f} \arrow{se,t}{\sigmat_f} \arrow{sw,t}{\pi_2}
\\
\node{\M_f}  \node[2]{\M_f}
\end{diagram}
\end{equation}
 commutes for some map ${\sigmat_f} \colon \M'_f \to
 \M_f$,
 \item[iii.] If  $\pi_1(\tau_1)=\pi_1(\tau_2)$ then $f_{\tau_1}=f_{\tau_2}$ up to pre- and post-composition by Moebius transformations.
 
 In particular, for every $m \in \M_f$ there are only finitely many different $f_\tau$ when $\tau \in \pi^{-1}(m)$.
\end{itemize}

\end{proposition}

\begin{remark} 
\label{rk:covers}
Note that $\M'_f$ is a quotient of $\T_f$ by a subgroup $G$ of the pure mapping class group of finite index. Then the quotient $\AM'_f$ of $\AT_f$ by the same subgroup will be a compactification of $\M'_f$. The covers $\pi_1$ and $\pi_2$ can be extended to  the corresponding augmented spaces so that $\pi_2\circ\pi_1=\pi$ still holds. As in the case of the compactified moduli space, we parametrize boundary strata by equivalence classes of multicurves, two classes of simpled closed curves being equivalent if one can be mapped to the other by the action of $G$. Evidently, the whole diagram above also extends to the augmented spaces.
\end{remark}

 Since the pure mapping class group acts by isometries with respect to both the \Teich and Weil-Petersson metrics, both metrics can be projected to $\M_f$ and $\M'_f$.

\begin{proofof}{\bf of Theorem~\ref{thm:Thurston} (Thurston's Theorem)}

{\bf Necessity of criterion.} Suppose $\Gamma$ is a Thurston obstruction. By Proposition~\ref{prop:fixedpts}, it is enough to show that $\sigma_f$ has no fixed points in $\T_f$. Suppose, by contradiction, that $\tau$ is a fixed point, we know then that every forward orbit must converge to $\tau$ (see Section~\ref{sec:Teich}). We may assume that $\Gamma$ is simple. By Proposition~\ref{prop:attracting}
the stratum $\S_\Gamma$ is weakly attracting, so we can choose an invariant neighborhood $U$ of $\S_\Gamma$ that does not contain $\tau$. But then orbits of points from $U$ are contained in $U$ and can not converge to $\tau$, which is a contradiction.

{\bf Sufficiency of criterion.}
Pick a point $\tau_0 \in \T_f$ and set $\tau_n=\sigma_f^n(\tau)$. Consider the projection of $\{\tau_n\}$ to  $\M'_f$: $m'_n=\pi_1(\tau_n)$. Let $D=d_T(\tau_0,\tau_1)$ be the \Teich distance between $\tau_0$ and $\tau_1$. Since $\sigma_f$ is weakly contracting with respect to the \Teich metric, $d_T(\tau_i,\tau_{i+1}) \le D$ for all $i$. Let $m'$ be an accumulation point of $\{m'_n\}$ in $\AM'_f$ that belongs to a stratum of minimal possible dimension. 

If $m' \in \M'_f$, then $\{\tau_n\}$ converges in $\T_f$. Indeed, by part \emph{iii} of  Proposition~\ref{prop:covers}, the norm of the coderivative $(d\sigma_f)^*=(f_\tau)_*$ depends only on $\pi_1(\tau)$. 
We know that  $\|(d\sigma^k_f)^*\|_T <1$, for some $k\in\N$, because $f$ is a  Thurston map with hyperbolic orbifold. For the sake of simplicity, we assume that $\|(d\sigma_f)^*\|_T <1$ (if not, we apply the same argument for $\sigma^k_f$).
 On the $\pi_1$-preimage of any compact subset $K$ of $\M'_f$, we then have $\sup_{\tau \in \pi^{-1}(K)} \|(d\sigma_f)^*\|_T <1$. It follows that the \Teich distance between $\tau_n$ and $\tau_{n+1}$ is contracted by some definite factor $\lambda<1$ by $\sigma_f$ (i.e. $d_T(\tau_{n+1},\tau_{n+2}) \le \lambda d_T(\tau_{n},\tau_{n+1})$) when $\pi_1(\tau_n) \in K$.
We take $K$ to be a closed ball of radius $r$ around $m'$ in $M'_f$. 
Then infinitely many $\tau_n$ are in $K$, thus $d_T(\tau_n,\tau_{n+1})$ tends to 0. 

Take $N$ such that the distance between $\tau_N$ and $\tau_{N+1}$ is smaller than $r(1-\lambda)/2$ and the distance between $m'_N$ and $m'$ is less than $r/2$.
Since $m'_N \in K$ we get  $d_T(\tau_{N+1},\tau_{N+2}) \le \lambda d_T(\tau_{N},\tau_{N+1}) <\lambda r(1-\lambda)/2$. Since $d_T(\tau_N,\tau_{N+1} )\sum_i \lambda^i <r/2$, we see by induction that $m'_{N+k} \in K$ and  $d_T(\tau_{N+k}, \tau_{N+k+1}) \le \lambda^k d_T(\tau_{N}, \tau_{N+1})$ for all $k\in \N$. Thus  $\{\tau_n\}$ converges at least as fast as a geometric series to a fixed point of $\sigma_f$. In this case, $f$ is Thurston equivalent to a rational function.


From now on we assume that $m' \in \S_{[\Gamma']}$ with $\Gamma' =\{\gamma'_1,\ldots,\gamma'_s\}\neq \emptyset$. Since we chose $m'$ on the stratum of minimal possible dimension, it follows that there exists $L \in (0,2\log(\sqrt{2}+1))$ such that for all $\tau_n$ there exist at most $s$ different simple closed geodesics of length less than $L$. The Collaring Lemma implies that these geodesics are mutually disjoint. Choose $L_1>0$ satisfying $L_1<e^{-D}L/d_f$ and $1/L_1>e^{D}(2/\pi+(d_f p_f+1)/L)$. Consider $n$ such that $l(\gamma_i,\tau_n)<L_1$ with $\gamma_i \in [\gamma'_i]$ for all $i=\overline{1,s}$. We claim that $\Gamma_n=\{\gamma_1,\ldots,\gamma_s\}$ is completely invariant. 

Indeed, since $\log l(\gamma,\tau)$ is $1$-Lipschitz, we have $e^{-D} l(\gamma,\tau_{n})\le l(\gamma,\tau_{n+1}) \le e^{D}l(\gamma,\tau_{n})$. On the other hand, every essential preimage of any $\gamma_i$ has length at most $d_f L_1<e^{-D}L$ so it must be homotopic to a curve in $\Gamma_n$; this proves  invariance of $\Gamma_n$. If some $\gamma_i$ were homotopic to no preimage of curves from $\Gamma_n$, then by Proposition~\ref{prop:estimate} we would get $1/l(\gamma_i,\tau_{n+1})<2/\pi+(d_f p_f+1)/L<e^{-D}/L_1<e^{-D}/l(\gamma_i,\tau_n)$; this proves  complete invariance.

Take a subsequence $\{m'_{n_k}\}$ that converges to $m'$ and such that for each $n_k$ there exist 
$\gamma_i \in [\gamma'_i]$ for $i=\overline{1,s}$ such that $l(\gamma_i,\tau_{n_k})<L_1$. Define $\Gamma_{n_k}$ as above for each $k$. Then for any pair $j,k$ there exists an element $g$ of the deck transformation group $G$ corresponding to the covering $\pi_1$ (see Proposition~\ref{prop:covers} and Remark~\ref{rk:covers}) such that $g(\Gamma_{n_k})=\Gamma_{n_j}$. Consider a pair of points $\tau$ and $g(\tau)$;  the commutative diagram~(\ref{diag:covers}) yields $$\pi(\sigma_f(g(\tau)))=\sigmat_f(\pi_1(g(\tau)))=\sigmat_f(\pi_1(\tau))=\pi(\sigma_f(\tau)).$$
Hence, there exists an element $h$ of the pure mapping class group such that $\sigma_f(g(\tau))=h(\sigma_f(\tau))$. Since both $\Gamma_{n_j}$ and $\Gamma_{n_k}$ are completely invariant, we get $$g(\Gamma_{n_k})=\Gamma_{n_j}=f^{-1}(\Gamma_{n_j})=f^{-1}(g(\Gamma_{n_k}))=h(f^{-1}(\Gamma_{n_k}))=h(\Gamma_{n_k}).$$
  This implies that $h(\gamma)=g(\gamma)$ for all $\gamma \in \Gamma_{n_k}$. It follows that $\Gamma_{n_k}$ have the same Thurston matrix $M$ for all $k$.

By assumption $\Gamma_{n_1}$ is not an obstruction. 
Select $T(L)$ and  a semi-norm on $\R^s$ as in Proposition~\ref{prop:seminorm}.
Consider the sets 
$$U(\Gamma_{n_k},T)=\left\{ \tau \in \T_f | \inf_{ \gamma \notin \Gamma_{n_k}} l(\gamma,\tau
    ) \ge L \:\: \text{and} \:\: \|Z(\Gamma_{n_k},\tau)\|>T\right\},$$    
where $T \ge e^{D}T(L)$ 
 and is large enough so that $\tau_{1} \notin U(\Gamma_{n_k},T)$ for all $k$ (this is possible since for $k=\overline{1,\infty}$  the set $Z(\Gamma_{n_k},\tau)$ is bounded  in $\R^s$ for all $\tau \in \T_f$, because there are only finitely many short curves on the Riemann surface corresponding to $\tau$). Since $\{m'_{n_k}\}$ converges to $m'$, we can pick the smallest $n=n_k$ such that $\tau_{n} \in U(\Gamma_{n},T)$. We use the Lipschitz condition again to get that $\tau_{n-1} \in U(\Gamma_n,Te^{-D})$. Proposition~\ref{prop:seminorm} yields $\|Z(\Gamma_n,\tau_{n})\|< \delta\|Z(\Gamma_n,\tau_{n-1})\|$ because $Te^{-D}\ge T(L)$; in particular, we see that $\tau_{n-1} \in U(\Gamma_{n},T)$. 
 Continuing by induction, we see that $\tau_{n_{k-1}} \in U(\Gamma_{n},T)$ which is a contradiction. Therefore, $\Gamma_{n_1}$ is a simple obstruction and we are done.
\qed\end{proofof}

\begin{proofof}{\bf of Theorems~\ref{thm:CanonicalObstruction} and \ref{thm:curvebound}}

As above, pick a point $\tau_0 \in \T_f$ and set $\tau_n=\sigma_f^n(\tau), m'_n=\pi_1(\tau_n)$. If $f$ is Thurston equivalent to a rational map then $\tau_n$ converges in $\T_f$ and hence $\Gamma_f=\emptyset$.

If $f$ is not Thurston equivalent to a rational map, then there exists an accumulation point  $m'$ of $\{m'_n\}$ on a boundary stratum $\S_{[\Gamma']}$ of $\M'_f$ of smallest dimension possible. As we have shown in the previous proof, if $\tau_n$ is close enough to the stratum $\S_\Gamma$ of $\T_f$ where $\Gamma \in [\Gamma']$ then $\Gamma$ is a simple obstruction. Proposition~\ref{prop:attracting} tells us that once $\tau_n$ is in a small neighborhood of $\S_\Gamma$  then $\tau_m$ stays in that neighborhood for all $m>n$. Therefore, the accumulation set of $\{m'_n\}$ must be a subset of $\overline{\S_{[\Gamma']}}$, moreover $l(\gamma,\tau_n) \to 0$ for all $\gamma \in \Gamma$. As $\S_{[\Gamma']}$ was a stratum of minimal dimensions to contain an accumulation point of $\{m'_n\}$, it follows that the accumulation set of $\{m'_n\}$ lies in a compact subset of $\S_{[\Gamma']}$, and thus $l(\gamma,\tau_n)>L$ for all $n=\overline{1,\infty}$ and $\gamma \notin \Gamma$ with some constant $L>0$. This shows that $\Gamma$ is the canonical obstruction for $f$ and proves the statement of Theorem~\ref{thm:curvebound}.
\qed\end{proofof}

\section{The Thurston Boundary of the \Teich space}
\label{sec:ThurstonBoundary}

For a more detailed introduction to the notion of the Thurston boundary we address the reader to \cite{IT}. Let $S$ be the set of all free homotopy classes of simple closed curves on $\Sphere \sm P_f$.
Then the function $l(\gamma, \tau) \colon S \times \T_f \to \R^+$ can be viewed as a map from $\T_f$ to $\R\P^S$. This map can be proven to be an analytic injection (cf. Section~\ref{sec:Aug}) with the image homeomorphic to a ball of the same dimension as $\T_f$. The boundary of this ball in $\R\P^S$ is called the Thurston boundary of $T_f$. Points of the Thurston boundary are represented by positive real-valued functions on $S$ where two functions correspond to the same point if and only if their ratio is constant. For example, for any $\gamma \in S$ the topological intersection number $\langle \gamma, \cdot \rangle$ is a function on $S$ corresponding to a point on the Thurston boundary and the set of all such points is dense. The Thurston boundary can be identified with the set $\mathcal{PMF}$ of projective measured foliations \cite{T}. We will use the following basic fact.

\begin{proposition}
\label{prop:convtb}
Suppose that for a sequence $\{\tau_n\}\in \T_f$, the lengths $\{l(\gamma,\tau_n)\} \to 0$ and $l(\delta,\tau_n)>\eps$ for   all $n\in\N$ and all $\delta$ that are not homotopic to $\gamma$, with some $\eps>0$. Then $\{\tau_n\} $ converges to the point $ \langle \gamma,\cdot\rangle$ in the Thurston boundary.
\end{proposition}

Thurston's pullback map $\sigma_f$ can be decomposed into a composition of two maps as follows.
Suppose $g \colon R_1 \to R_2$ is a covering map between two surfaces $R_1$ and $R_2$ of finite type. Then one can define the usual pullback map $g^* \colon \T(R_2) \to \T(R_1)$. If  $i \colon R_1 \to R_2$ is an inclusion map between two surfaces $R_1$ and $R_2$ of finite type (that is a map that fills in some of the punctures of $R_1$) then one can define the push-forward map $i_* \colon \T(R_1) \to \T(R_2)$ (also called the forgetful map) that just forgets the information about the erased punctures. In our setting, we have $\sigma_f =i_* \circ g^*$ where $g=f|_{\Sphere \sm f^{-1}(P_f)}$ and $i=\id|_{\Sphere \sm f^{-1}(P_f)}$.

It is evident that the action of $g^* \colon \T_f \to \T(\Sphere,f^{-1}(P_f))$ could be continuously extended to the Thurston boundary using the natural action of $g^*$ on measured foliations. However, there is no natural way to push measured foliations forward. 

\begin{proposition}
 If $p_f>3$ then $i_*$ has no continuous extension to the Thurston boundary of $\T(\Sphere,f^{-1}(P_f))$.
\end{proposition}

\begin{proof}
Take a simple closed curve $\gamma$ in $\Sphere \sm f^{-1}(P_f)$ that separates two points $A$ and $B$ of $f^{-1}(P_f)$  such that $A \in P_f$ and $B \notin P_f$ from all other points of $f^{-1}(P_f)$. Connect $A$ and $B$ by a simple path $\delta$ that does not intersect $\gamma$. Fix a complex structure $\tau$ on $\Sphere \sm f^{-1}(P_f)$ and start moving $A$ towards $B$ along $\delta$. The obtained path $\delta_1$ in $\T(\Sphere,f^{-1}(P_f))$ tends to the point on the Thurston boundary defined by $\langle \gamma, \cdot \rangle$ by Proposition~\ref{prop:convtb}. Indeed, the length of any curve $\alpha$ that has zero intersection number with $\gamma$ is bounded below by the length of $\alpha$ on the Riemann surface obtained from $\tau$ by filling in the puncture at $A$, and the length of $\gamma$ clearly tends to 0.

 If we depart from the same initial complex structure and start moving $B$ to $A$ along the path $\delta$, we get a new path $\delta_2$ in $\T(\Sphere,f^{-1}(P_f))$ with the same limit. It is clear that the limits of $i_*(\delta_1)$ and $i_*(\delta_2)$ are different in $\T_f$.
\qed\end{proof}

The previous proposition is the moral reason why $\sigma_f$ can not be extended to the Thurston boundary. But since the image of $g^*$ is by far not the whole of $\T(\Sphere,f^{-1}(P_f))$ (it can not have dimension greater than the dimension of $\T_f$), we have to say a little bit more. If we assume that $\sigma_f$ extends continuously  to the Thurston boundary, we get the following  necessary condition on $f$.

\begin{proposition} Suppose that $\sigma_f$ extends continuously  to the Thurston boundary. If for some essential simple closed curve $\gamma$ in $\Sphere \sm P_f$, all $f$-preimages of $\gamma$ are non-essential, then $\sigma_f$ is constant on the stratum $\S_{\{\gamma\}}$.
\label{prop:necessarycond}
\end{proposition}

\begin{proof}
 By Proposition~\ref{prop:convtb}, any sequence $\{\tau_n\} \in \T_f$ that converges to a point on $\S_{\{\gamma\}}$ in the augmented \Teich space also converges to  $ \langle \gamma,\cdot\rangle$ in the Thurston compactification. Since the stratum $\S_{\{\gamma\}}$ is mapped into $\T_f$ by Proposition~\ref{prop:invariantstrata}, we see that $\sigma_f$ must be constant on it. Indeed, for any $\tau \in \S_{\{\gamma\}}$, we can consider a sequence $\{\tau_n\}\in\T_f$ converging to $\tau$. We get $\sigma_f(\tau)=\lim \sigma_f(\tau_n) =\sigma_f( \langle \gamma,\cdot\rangle)$. \qed
\end{proof}

We conclude by constructing explicit examples where the previous condition is violated.

\begin{theorem} 
\label{thm:noextension}
There exist Thurston maps $f$ such that Thurston's pullback map does not extend to the Thurston boundary of the \Teich space.
\end{theorem}

\begin{proof}
 We start with a Thurston map $(f,P_f)$ which is a topological polynomial (i.e. there exists a fixed point $\infty \in P_f$ that has no $f$-preimages other than itself)  with non-constant $\sigma_f$. We may assume that $f$ has  two fixed points $A$ and $B$ outside of $P_f$ because otherwise we can create extra fixed points by applying a homotopy relative to $P_f$. Consider the Thurston map $f'=(f,P_f\cup\{A,B\})$.

Let $\gamma$ be a simple closed curve that  separates points $A$ and $B$ from $P_f$. A component $\delta$ of the $f$-preimage  of $\gamma$ is essential if and only if it also separates $A$ and $B$ from $P_f$ because the complimentary component of $\delta$ that does not contain $\infty$ can contain no marked points except $A$ and $B$. Thus, if we assume that all curves that  separate points $A$ and $B$ from $P_f$ have  essential preimages, we  easily get that $f$ is a homeomorphism. We fix such a $\gamma$ that has no essential preimages. 

Denote by $P_A:\T(\Sphere, P_f\cup\{A,B\} )\to \T(\Sphere, P_f\cup\{B\} )$ and $P_{AB}:\T(\Sphere, P_f\cup\{A,B\} )\to \T(\Sphere, P_f)$ the  canonical projections between the respective \Teich spaces. Note that $\S_{\{\gamma\}}$ is canonically isomorphic to  $\T(\Sphere, P_f\cup\{B\}$. We get the following commutative diagram. 

$$
\begin{diagram}
\node{\S_{\{\gamma\}}=\T(\Sphere, P_f\cup\{B\})} \arrow{s,l}{P_A} \arrow{e,t}{ \sigma_{f'}} \node{\T(\Sphere, P_f\cup\{A,B\})} \arrow{s,r}{P_{AB}}
\\
\node{\T(\Sphere, P_f)} \arrow{e,t}{ \sigma_{f}} \node{\T(\Sphere, P_f)}
\end{diagram}
$$

Since $\sigma_f$ is non-constant by the assumption and $P_A$ is surjective, the map $\sigma_{f'}$ is non-constant on $\S_{\{\gamma\}}$.  Proposition~\ref{prop:necessarycond} implies that $\sigma_{f'}$ can not be continuously extended to the Thurston boundary of the \Teich space.
 \qed
\end{proof}

\section{Pilgrim's conjecture}
\label{sec:Pilgrim}

The geometric description of the extension of $\sigma_f$ to the boundary of the augmented \Teich space allows us to understand fairly well what exactly happens when $f$ is obstructed. In this case, by Theorem~\ref{thm:CanonicalObstruction}  the sequence $\sigma_f^n(\tau)$ tends to $\S_{\Gamma_f}$ for all $\tau \in \T_f$, where $\Gamma_f$ is the canonical obstruction for $f$. 
Recall that the action on any invariant stratum is given by pullbacks of complex structures by a collection of maps $\sigma_{f^C}$ for all components $C$ of any surface in the stratum. As we iterate, a new complex structure on each component is obtained by pulling back an old one from the image component. The combinatorics of the process is very simple: we have a map from a finite set into itself, every component is (pre-)periodic. The whole action, therefore, can be characterized by studying cycles of components. For each component $C$ there are three cases (compare \cite[Canonical Decomposition Theorem]{P1}): the composition $F:=F^C$ of all coverings in the cycle (that is the first-return map of $C$) is one of the following:
\begin{itemize}
 \item a homeomorphism,
 \item a Thurston map with parabolic orbifold,
 \item a Thurston map with hyperbolic orbifold.
\end{itemize}

 In the first case, the map $\sigma_{F}$ acts on $\T_F$ as an element of the mapping class group. All possible Thurston maps with parabolic orbifold are described in Section~9 of \cite{DH}. We prove below that in the stratum $\S_{\Gamma_f}$ all Thurston maps with hyperbolic orbifold are not obstructed.
 
\begin{remark} We consider the first-return map $F$ to be a Thurston map in the sense of the definition given in Remark~\ref{rk:GenDef} where we take $P$ to be the set of marked points (including nodes) of $C$. Clearly, if $(F,P)$ is not obstructed then neither is $(F,P_F)$.
\end{remark}
 
 By considering $f^r$ where $r$ is the least common multiple of the lengths of all cycles of components, we can assume that all components of $C$ are fixed or prefixed. Clearly $P_f=P_{f^r}$ (which means that $\T_f=\T_{f^r}$), $\sigma_{f^r}=\sigma^r_f$, and any  $f$-invariant multicurve $\Gamma$ is $f^r$-invariant; the Thurston matrix $M_\Gamma$ for $f^r$ is the $r$-th power of the analogous Thurston matrix for $f$ (see Lemma 1.1 in \cite{DH}). Moreover, the following is immediate.
 
\begin{proposition}
$\Gamma_f=\Gamma_{f^r}$.
\end{proposition}

\begin{proof} Take any $\tau \in \T_f$. If $\gamma \in \Gamma_f$ then $l(\gamma, \sigma^n_f(\tau) ) \to 0$. In particular, $l(\gamma,\sigma^{nr}_f(\tau) ) = l(\gamma,\sigma^{n}_{f^r}(\tau) )\to 0$, hence $\gamma \in \Gamma_{f^n}$. 

Set $D=d_T(\tau,\sigma_f(\tau))$. Since $\sigma_f$ is weakly contracting and $\log l(\cdot, \gamma)$ is 1-Lipschitz, we have $l (\gamma,\sigma^{nr+s}_f(\tau)) \le e^{Ds} l( \gamma,\sigma^{nr}_f(\tau))$. Therefore, if $l( \gamma,\sigma^{nr}_f(\tau)) \to 0$ then  $l( \gamma,\sigma^n_f(\tau)) \to 0$ as well.
\qed\end{proof}

 Using the tools developed, we are now able to give a proof of the following conjecture from \cite{P1}.
 
\begin{theorem}
\label{thm:KevinConjecture}
  If the first-return map $F$ of a periodic component $C$ of the topological surface corresponding to the stratum $\S_{\Gamma_f}$ has hyperbolic orbifold then $F$ is not obstructed and, hence, equivalent to a rational map.
\end{theorem}

\begin{proof}

   We start again by considering the orbit of a point $\tau_0 \in \T_f$ and denote $\tau_n=\sigma_f^n(\tau), m'_n=\pi_1(\tau_n)$, and $D=d_T(\tau_0,\tau_1)$. Theorems~\ref{thm:CanonicalObstruction} and \ref{thm:curvebound} imply that the limit set of $m'_n$ is contained in a compact subset of $\S_{[\Gamma_{f}]} \subset \AM'_f$.
  
  Any point $\th$ that lies in the stratum $\S_{\Gamma_f}$ can be represented as $\th=(\th^1,\ldots,\th^s)$ where $\th^i$ are points in the \Teich spaces corresponding to different components of the noded topological surface corresponding to the stratum; let us say that $\th^C:=\th^1$ is the coordinate corresponding to $C$, i.e. a point in $\T_F$. We similarly write $\mh=(\mh^C,\ldots,\mh^s)$ for points in $\S_{[\Gamma_f]}$. As was explained above, we may assume that $C$ is a fixed component. The action of $\sigma$ on $\S_{\Gamma_f}$ can then be written in a form $\sigma_f^i(\th)=(\sigma_F^i(\th^C),\ldots)$. For notational convenience, we assume that $\|d\sigma_{F}\|<1$; otherwise we can take an appropriate iterate of $f$ and work from there. Denote by $\pi^C,\pi_1^C,\pi_2^C$ and $\sigmat_{F}$ the maps that we get applying Proposition~\ref{prop:covers} to $F$ where $\M'_{F}$ is the ``restriction'' of $\M'_f$ to the component corresponding to $C$ (we chose  $\M'_{F}$ so that $\pi_1^C$ is a restriction of $\pi_1$; we may be able to construct a smaller covering space for $F$ that satisfies conditions of Propostion~\ref{prop:covers}).
  
    To illustrate the idea of the proof, let us first suppose that there exists an accumulation point $\tt \in \S_{\Gamma_f}$ of $\{\tau_n\}$. Then, clearly, $\tt_i:=\sigma_f^i(\tt)$ is also in the accumulation set of $\{\tau_n\}$ for all $i$. Hence, $\{\pi_1(\tt_i)\}$  is contained in a compact subset of $\S_{[\Gamma_{f}]}$. Considering the coordinate corresponding to $C$, we get that $\pi_1^C(\tt_i^C)=\pi_1^C(\sigma_F^i(\tt^C))$ lie in a compact subset of $\M'_{F}$. Theorem~\ref{thm:CanonicalObstruction} would imply that $F$ is not obstructed. In the rest of the proof we will remove the assumption we made. For each $t\in\N$ we find a point $\th_k \in \S_{\Gamma_{f}}$ such that the first $t$ elements of the sequence $\{\pi_1(\sigma_f^i(\th_k))\}$ coincide with the corresponding elements of some fixed sequence contained in a compact subset of $\M'_{f}$. From this we will be able to conclude the proof.
  
  Choose a sequence $\{n_k^1\}$ so that $m'_{n_k^1} \to \mh_0 \in \S_{[\Gamma_{f}]}$, then choose a subsequence $\{n_k^2\}$ of $\{n_k^1\}$ so that $m'_{n_k^2+1} \to \mh_1$ and so on, so that  $m'_{n_k^i+j} \to \mh_j$ for all $0 \le j \le i$. Then the diagonal subsequence $\{n_k:=n_k^k\}$ satisfies $m'_{n_k^k+i} \to \mh_i$ for all $i=\overline{0,\infty}$. The commutative diagram~(\ref{diag:covers}) implies $\sigmat_f(m'_{n_k+i})= \pi_2(m'_{n_k+i+1})$ so by continuity (Theorem~\ref{thm:extensionA}) we get
\begin{equation}
	\sigmat_f(\mh_i)= \pi_2(\mh_{i+1}).
\label{eq:2}
\end{equation} 

  Suppose $d_{WP}(m'_{n_k},\mh_0)<\eps$ and $d_{WP}(m'_{n_k+1},\mh_1)<\eps$  for some $k$, where $\eps>0$ is small. Then there exists a point $\th_k \in \S_{\Gamma_f}$ such that $d_{WP}(\tau_{n_k},\th_k)=d_{WP}(m'_{n_k},\mh_0)<\eps$ and $\pi_1(\th_k)=\mh_0$. Then Proposition~\ref{prop:Lipschitz} implies 
 \begin{align*}
d_{WP}(\mh_1,\pi_1(\sigma_f(\th_k))) \le 
   &  d_{WP}(m'_{n_k+1},\pi_1(\sigma_f(\th_k))) + d_{WP}(\mh_1,m'_{n_k+1}) \le
    \\
  & \le d_{WP}(\tau_{n_k+1},\sigma_f(\th_k)) + \eps \le (\sqrt{d_f}+1)\eps.
\end{align*}
  
 Since $\pi(\sigma_f(\th_k))=\sigmat_f(\mh_0)=\pi_2(\mh_1)$ we have that $\pi_1(\sigma_f(\th_k))$ lies in the fiber $\pi_2^{-1}(\pi_2(\mh_1))$. 
Since $\pi_2^{-1}(\pi_2(\mh_1))$ is finite in $\M'_f$ there exists a positive constant $c_1$ such that the Weil-Petersson distance between any two different points in the fiber is at least $c_1$. Similarly, set $c_i$ to be the minimal distance  between any two different points in the fiber $\pi_2^{-1}(\pi_2(\mh_i))$ for all $i=\overline{1,\infty}$.  We conclude that $\pi_1(\sigma_f(\th_k))=\mh_1$ if $\eps \le c_1/(\sqrt{d_f}+1)$.

For any positive integer $t$, we set $\eps(t)=\min_{i=\overline{1,t}} \{c_i\}/(\sqrt{d_f}+1)$ and chose $k=k(t)$ large enough so that $d_{WP}(m'_{n_k+i},\mh_i)<\eps(t)$ for all $i=\overline{1,t}$. 
Then the same reasoning yields
\begin{equation}
\pi_1(\sigma_f^i(\th_k))=\mh_i
\label{eq:1}
\end{equation}
 for all $i=\overline{1,t}$.

 By the first part of Proposition~\ref{prop:metriccompairson} $d_{WP}(\tau_{n_k},\tau_{n_k+1})\le C^0 d_{T}(\tau_{n_k},\tau_{n_k+1})$ and therefore
 \begin{align*}
	d_{WP}(\th_k,\sigma_f(\th_k)) &\le
	d_{WP}(\th_k,\tau_{n_k}) + d_{WP}(\tau_{n_k},\tau_{n_k+1})  + d_{WP}(\tau_{n_k+1},\sigma_f(\th_k)) \le \\
	&\le \eps(t) + C^0\cdot D + \sqrt{d_f}\eps(t) \le c_1+C^0\cdot D.
 \end{align*} 
 is bounded independently of $k$. The length $L$ of the shortest simple closed geodesic  on the Riemann surface corresponding to $\th^C_k$ is the same for all $k$ since $\pi(\th_k)=\pi_2(\mh_0)$ does not depend on $k$. Using the second part of Proposition~\ref{prop:metriccompairson} 
 we obtain 
 \begin{align*}
 &d_{T}(\th_k^C,\sigma_{F}(\th_k^C)) \le C^1(L) d_{WP}(\th_k^C,\sigma_{F}(\th_k^C))\le \\ &\le  C^1(L) d_{WP}(\th_k,\sigma_f(\th_k))\le C^1(L)(c_1+C^0 \cdot D)=:D_1
 \end{align*}
  where $D_1$ is a constant (note that the first two distances are measured in $\T_F$; the second inequality follows from the definition of the Weil-Petersson metric on the boundary, see Section~\ref{sec:Aug}).  Then (\ref{eq:1}) implies $\pi_1^C(\sigma_{F}^i(\th_k^C))=\mh_i^C$ for $i=\overline{1,t}$, thus this sequence lies in a compact subset of $\M'_{F}$. From the fact that $\sigma_{F}$ is uniformly contracting on the $\pi_1^C$-preimage of any compact set in $\M'_{F}$, it follows that $d_T(\mh_i^C,\mh_{i+1}^C) \le d_T(\sigma_{F}^i(\th_k^C),\sigma_{F}^{i+1}(\th_k^C)) \le q^i D_1$ for some $q<1$ and all $i=\overline{1,t}$. 
 
 Since the bound $D_1$ does not depend on $k$, the inequality $d_T(\mh_i^C,\mh_{i+1}^C) \le q^i D_1$ follows for all $i$ and, hence, $\{\mh_i^C\}$ converges to some  $\mh^C \in \M'_{F}$. It follows from (\ref{eq:2}) that $\sigmat_{F}(\mh^C)=\pi_2^C(\mh^C)$ since all the maps involved are continuous. This means that for any point $\tau$ in the fiber $(\pi_1^C)^{-1}(\mh^C)$ both $\tau$ and $\sigma_{F}(\tau)$ lie in $A=(\pi^C)^{-1}(\pi_2^C(\mh^C))$. Let $R>0$ be the lower bound on the \Teich distance in $\T_F$ between distinct points in the fiber $A$.  
 
 Choose $t$ and $\eps$ such that $2\eps + q^tD_1<R$ and $d_{T}(\mh^C,\mh^C_t) < \eps$; set $k=k(t)$. 
 Since $\pi_1^C(\sigma_{F}^t(\th_k^C))= \mh^C_t$ we can find a point $\th^C \in \T_{F}$ such that $\pi_1^C(\th^C)=\mh^C$ and $d_{T}(\th^C,\sigma_{F}^t(\th_k^C))=d_{T}(\mh^C,\mh^C_t) < \eps$.
 Then 
 \begin{align*}
 &d_T(\th^C,\sigma_{F}(\th^C)) \le \\ \le
 d_T(\th^C,\sigma_{F}^t(\th_k^C)) + &d_T(\sigma_{F}^t(\th_k^C),\sigma_{F}^{t+1}(\th_k^C))  + d_T(\sigma_{F}^{t+1}(\th_k^C),\sigma_{F}(\th^C)) \le \\
 &\le \eps + q^tD_1 + \eps < R.
\end{align*}
 Since both $\th^C$ and $\sigma_{F}(\th^C)$ belong to $A$, it follows that $\sigma_{F}(\th^C)=\th^C$ yielding that $F$ is equivalent to a rational map.
\qed\end{proof}

As a corollary we have the following

\begin{theorem}
\label{thm:fixedpts}
If the first-return maps of all periodic components of the topological surface corresponding to the stratum $\S_{\Gamma_f}$ have hyperbolic orbifolds, then
 $\sigma_f$ has a unique fixed point $\th$ in this stratum, and the orbit of any point in $\T_f$ converges to $\th$.

\end{theorem}
  
\begin{proof} Take any point $\th_0$ in $\S_{\Gamma_f}$ and consider its forward orbit. Since all $F^C$ are not obstructed, the sequence will tend to a limit $\th$ in $\S_{\Gamma_f}$ which    is a fixed point by Theorem~\ref{thm:extensionA}. Indeed, if $C$ is a fixed component then we deal with Thurston's pullback map for a branched cover $F^C$ which has hyperbolic orbifold and is not obstructed. Therefore the coordinate corresponding to this component converges to the unique fixed point of $F^C$. If $C$ is in a cycle of components of length $n$, then by the same argument the coordinate corresponding to $C$ for the sequence $\{\sigma_f^{nk+i}(\th_0)\}$ converges to the unique fixed point of $F^C$ for any given $i$ as $k$ goes to infinity. Thus the coordinate corresponding to $C$ converges for the whole sequence $\sigma_f^{k}(\th_0)$.
Convergence of coordinates corresponding to pre-periodic components follows then from continuity of Thurston's pullback map.

To see that every orbit in $\T_f$ converges to $\th$, note that from the proof of the previous theorem, it follows that $\th$ is in the limit set of any orbit in $\T_f$. On the other hand, it is easy to see that $\th$ is weakly attracting in the sense of the definition given in Section~\ref{sec:strata}, therefore the orbit must converge to $\th$.
\qed\end{proof}

\begin{acknowledgements}
This article is an extended and completely reworked version of the research announcement \cite{S}. I thank all the people who helped and motivated me, including Dima Dudko, Adam Epstein, Alan Huckleberry, Sarah Koch, Daniel Meyer, Kevin Pilgrim, Dierk Schleicher. I thank the anonymous referee for many useful comments. This research is financially supported by Deutsche Forschungsgemeinschaft. I also thank the German-Israeli foundation and CODY network for additional support.
\end{acknowledgements}

\bibliographystyle{amsplain}
\bibliography{selinger}

\providecommand{\bysame}{\leavevmode\hbox to3em{\hrulefill}\thinspace}
\providecommand{\MR}{\relax\ifhmode\unskip\space\fi MR }
\providecommand{\MRhref}[2]{%
  \href{http://www.ams.org/mathscinet-getitem?mr=#1}{#2}
}
\providecommand{\href}[2]{#2}
\begin{thebibliography}{BEKP09}

\bibitem[AB60]{AB}
Lars Ahlfors and Lipman Bers, \emph{Riemann's mapping theorem for variable
  metrics}, Ann. of Math. (2) \textbf{72} (1960), 385--404. \MR{0115006 (22
  \#5813)}

\bibitem[BEKP09]{BEKP}
Xavier Buff, Adam Epstein, Sarah Koch, and Kevin Pilgrim, \emph{On {T}hurston's
  pullback map}, Complex dynamics, A K Peters, Wellesley, MA, 2009,
  pp.~561--583. \MR{2508269 (2010g:37071)}

\bibitem[Ber74a]{B2}
Lipman Bers, \emph{On spaces of {R}iemann surfaces with nodes}, Bull. Amer.
  Math. Soc. \textbf{80} (1974), 1219--1222. \MR{0361165 (50 \#13611)}

\bibitem[Ber74b]{B1}
\bysame, \emph{Spaces of degenerating {R}iemann surfaces}, Discontinuous groups
  and {R}iemann surfaces ({P}roc. {C}onf., {U}niv. {M}aryland, {C}ollege
  {P}ark, {M}d., 1973), Princeton Univ. Press, Princeton, N.J., 1974,
  pp.~43--55. Ann. of Math. Studies, No. 79. \MR{0361051 (50 \#13497)}

\bibitem[DH93]{DH}
Adrien Douady and John~H. Hubbard, \emph{A proof of {T}hurston's topological
  characterization of rational functions}, Acta Math. \textbf{171} (1993),
  no.~2, 263--297. \MR{1251582 (94j:58143)}

\bibitem[Hub06]{H}
John~Hamal Hubbard, \emph{Teichm\"uller theory and applications to geometry,
  topology, and dynamics. {V}ol. 1}, Matrix Editions, Ithaca, NY, 2006,
  Teichm{\"u}ller theory, With contributions by Adrien Douady, William Dunbar,
  Roland Roeder, Sylvain Bonnot, David Brown, Allen Hatcher, Chris Hruska and
  Sudeb Mitra, With forewords by William Thurston and Clifford Earle.
  \MR{2245223 (2008k:30055)}

\bibitem[IT92]{IT}
Y.~Imayoshi and M.~Taniguchi, \emph{An introduction to {T}eichm\"uller spaces},
  Springer-Verlag, Tokyo, 1992, Translated and revised from the Japanese by the
  authors. \MR{1215481 (94b:32031)}

\bibitem[KH95]{KH}
Anatole Katok and Boris Hasselblatt, \emph{Introduction to the modern theory of
  dynamical systems}, Encyclopedia of Mathematics and its Applications,
  vol.~54, Cambridge University Press, Cambridge, 1995, With a supplementary
  chapter by Katok and Leonardo Mendoza. \MR{1326374 (96c:58055)}

\bibitem[Mas76]{Masur}
Howard Masur, \emph{Extension of the {W}eil-{P}etersson metric to the boundary
  of {T}eichmuller space}, Duke Math. J. \textbf{43} (1976), no.~3, 623--635.
  \MR{0417456 (54 \#5506)}

\bibitem[McM00]{McM}
Curtis~T. McMullen, \emph{The moduli space of {R}iemann surfaces is {K}\"ahler
  hyperbolic}, Ann. of Math. (2) \textbf{151} (2000), no.~1, 327--357.
  \MR{1745010 (2001m:32032)}

\bibitem[Mil06]{M}
John Milnor, \emph{Dynamics in one complex variable}, third ed., Annals of
  Mathematics Studies, vol. 160, Princeton University Press, Princeton, NJ,
  2006. \MR{2193309 (2006g:37070)}

\bibitem[Pil01]{P}
Kevin~M. Pilgrim, \emph{Canonical {T}hurston obstructions}, Adv. Math.
  \textbf{158} (2001), no.~2, 154--168. \MR{1822682 (2001m:57004)}

\bibitem[Pil03]{P1}
\bysame, \emph{Combinations of complex dynamical systems}, Lecture Notes in
  Mathematics, vol. 1827, Springer-Verlag, Berlin, 2003. \MR{2020454
  (2004m:37087)}

\bibitem[Sel09]{S}
Nikita Selinger, \emph{On the boundary behavior of {T}hurston's pullback map},
  Complex dynamics, A K Peters, Wellesley, MA, 2009, pp.~585--595. \MR{2508270
  (2010g:37068)}

\bibitem[Thu88]{T}
William~P. Thurston, \emph{On the geometry and dynamics of diffeomorphisms of
  surfaces}, Bull. Amer. Math. Soc. (N.S.) \textbf{19} (1988), no.~2, 417--431.
  \MR{956596 (89k:57023)}

\bibitem[Wol03]{W1}
Scott~A. Wolpert, \emph{Geometry of the {W}eil-{P}etersson completion of
  {T}eichm\"uller space}, Surveys in differential geometry, {V}ol.\ {VIII}
  ({B}oston, {MA}, 2002), Surv. Differ. Geom., VIII, Int. Press, Somerville,
  MA, 2003, pp.~357--393. \MR{2039996 (2005h:32032)}

\bibitem[Wol09]{W2}
\bysame, \emph{The {W}eil-{P}etersson metric geometry}, Handbook of
  {T}eichm\"uller theory. {V}ol. {II}, IRMA Lect. Math. Theor. Phys., vol.~13,
  Eur. Math. Soc., Z\"urich, 2009, pp.~47--64. \MR{2497791 (2010i:32012)}

\end{thebibliography}


\providecommand{\bysame}{\leavevmode\hbox to3em{\hrulefill}\thinspace}
\providecommand{\MR}{\relax\ifhmode\unskip\space\fi MR }
\providecommand{\MRhref}[2]{%
  \href{http://www.ams.org/mathscinet-getitem?mr=#1}{#2}
}
\providecommand{\href}[2]{#2}
\begin{thebibliography}{10}

\bibitem{abi}
William Abikoff, \emph{Degenerating families of {R}iemann surfaces}, Ann. of
  Math. (2) \textbf{105} (1977), no.~1, 29--44. \MR{0442293 (56 \#679)}

\bibitem{AB}
Lars Ahlfors and Lipman Bers, \emph{Riemann's mapping theorem for variable
  metrics}, Ann. of Math. (2) \textbf{72} (1960), 385--404. \MR{0115006 (22
  \#5813)}

\bibitem{B2}
Lipman Bers, \emph{On spaces of {R}iemann surfaces with nodes}, Bull. Amer.
  Math. Soc. \textbf{80} (1974), 1219--1222. \MR{0361165 (50 \#13611)}

\bibitem{B1}
\bysame, \emph{Spaces of degenerating {R}iemann surfaces}, Discontinuous groups
  and {R}iemann surfaces ({P}roc. {C}onf., {U}niv. {M}aryland, {C}ollege
  {P}ark, {M}d., 1973), Princeton Univ. Press, Princeton, N.J., 1974,
  pp.~43--55. Ann. of Math. Studies, No. 79. \MR{0361051 (50 \#13497)}

\bibitem{BEKP}
Xavier Buff, Adam Epstein, Sarah Koch, and Kevin Pilgrim, \emph{On {T}hurston's
  pullback map}, Complex dynamics, A K Peters, Wellesley, MA, 2009,
  pp.~561--583. \MR{2508269 (2010g:37071)}

\bibitem{TanLei}
Xavier Buff, Cui Guizhen, and Tan Lei, \emph{Teich\"uller spaces and
  holomorphic dynamics}, Handbook of Teichm\"uller theory, Vol. III, to appear.

\bibitem{DH}
Adrien Douady and John~H. Hubbard, \emph{A proof of {T}hurston's topological
  characterization of rational functions}, Acta Math. \textbf{171} (1993),
  no.~2, 263--297. \MR{1251582 (94j:58143)}

\bibitem{H}
John~Hamal Hubbard, \emph{Teichm\"uller theory and applications to geometry,
  topology, and dynamics. {V}ol. 1}, Matrix Editions, Ithaca, NY, 2006,
  Teichm{\"u}ller theory, With contributions by Adrien Douady, William Dunbar,
  Roland Roeder, Sylvain Bonnot, David Brown, Allen Hatcher, Chris Hruska and
  Sudeb Mitra, With forewords by William Thurston and Clifford Earle.
  \MR{2245223 (2008k:30055)}

\bibitem{IT}
Y.~Imayoshi and M.~Taniguchi, \emph{An introduction to {T}eichm\"uller spaces},
  Springer-Verlag, Tokyo, 1992, Translated and revised from the Japanese by the
  authors. \MR{1215481 (94b:32031)}

\bibitem{KH}
Anatole Katok and Boris Hasselblatt, \emph{Introduction to the modern theory of
  dynamical systems}, Encyclopedia of Mathematics and its Applications,
  vol.~54, Cambridge University Press, Cambridge, 1995, With a supplementary
  chapter by Katok and Leonardo Mendoza. \MR{1326374 (96c:58055)}

\bibitem{Masur}
Howard Masur, \emph{Extension of the {W}eil-{P}etersson metric to the boundary
  of {T}eichm\"uller space}, Duke Math. J. \textbf{43} (1976), no.~3, 623--635.
  \MR{0417456 (54 \#5506)}

\bibitem{McM}
Curtis~T. McMullen, \emph{The moduli space of {R}iemann surfaces is {K}\"ahler
  hyperbolic}, Ann. of Math. (2) \textbf{151} (2000), no.~1, 327--357.
  \MR{1745010 (2001m:32032)}

\bibitem{M}
John Milnor, \emph{Dynamics in one complex variable}, third ed., Annals of
  Mathematics Studies, vol. 160, Princeton University Press, Princeton, NJ,
  2006. \MR{2193309 (2006g:37070)}

\bibitem{P}
Kevin~M. Pilgrim, \emph{Canonical {T}hurston obstructions}, Adv. Math.
  \textbf{158} (2001), no.~2, 154--168. \MR{1822682 (2001m:57004)}

\bibitem{P1}
\bysame, \emph{Combinations of complex dynamical systems}, Lecture Notes in
  Mathematics, vol. 1827, Springer-Verlag, Berlin, 2003. \MR{2020454
  (2004m:37087)}

\bibitem{S}
Nikita Selinger, \emph{On the boundary behavior of {T}hurston's pullback map},
  Complex dynamics, A K Peters, Wellesley, MA, 2009, pp.~585--595. \MR{2508270
  (2010g:37068)}

\bibitem{T}
William~P. Thurston, \emph{On the geometry and dynamics of diffeomorphisms of
  surfaces}, Bull. Amer. Math. Soc. (N.S.) \textbf{19} (1988), no.~2, 417--431.
  \MR{956596 (89k:57023)}

\bibitem{W1}
Scott~A. Wolpert, \emph{Geometry of the {W}eil-{P}etersson completion of
  {T}eichm\"uller space}, Surveys in differential geometry, {V}ol.\ {VIII}
  ({B}oston, {MA}, 2002), Surv. Differ. Geom., VIII, Int. Press, Somerville,
  MA, 2003, pp.~357--393. \MR{2039996 (2005h:32032)}

\bibitem{W2}
\bysame, \emph{The {W}eil-{P}etersson metric geometry}, Handbook of
  {T}eichm\"uller theory. {V}ol. {II}, IRMA Lect. Math. Theor. Phys., vol.~13,
  Eur. Math. Soc., Z\"urich, 2009, pp.~47--64. \MR{2497791 (2010i:32012)}

\end{thebibliography}

\end{document}